\documentclass[12pt]{article} 
\usepackage[sectionbib]{natbib}
\usepackage{fancyhdr}
\usepackage{array,epsfig,rotating}
\RequirePackage[colorlinks,citecolor=blue,urlcolor=blue]{hyperref}

\usepackage{sectsty, secdot}
\RequirePackage{amsthm,amsmath,amsfonts,amssymb}
\RequirePackage{graphicx}
\usepackage{makecell}
\usepackage{etoc}
\usepackage{setspace}

\usepackage{tikz}
\usepackage[most]{tcolorbox}
\usepackage{pgfplots}
\pgfplotsset{compat=1.16}
\usepackage{color}

\sectionfont{\fontsize{12}{14pt plus.8pt minus .6pt}\selectfont}
\renewcommand{\theequation}{\thesection\arabic{equation}}
\subsectionfont{\fontsize{12}{14pt plus.8pt minus .6pt}\selectfont}
\usepackage[a4paper,margin=1in,footskip=0.25in]{geometry}
\usepackage{amsmath}
\usepackage{amssymb}
\usepackage{amsfonts}
\usepackage{multirow}
\usepackage{amsthm}

\newtheorem{theorem}{Theorem}
\newtheorem{lemma}{Lemma}
\newtheorem{corollary}{Corollary}

\theoremstyle{definition}

\usepackage{etoc}

\newtheorem{prop}[theorem]{Proposition}

\newif\ifsupp

\supptrue
\begin{document}

\renewcommand{\baselinestretch}{2}

\markright{ \hbox{\footnotesize\rm Statistica Sinica
}\hfill\\[-13pt]
\hbox{\footnotesize\rm
}\hfill }

\markboth{\hfill{\footnotesize\rm Alon Kipnis} \hfill}
{\hfill {\footnotesize\rm Rare/Weak models and log-chisquared P-values} \hfill}

\renewcommand{\thefootnote}{}
$\ $\par


\fontsize{12}{14pt plus.8pt minus .6pt}\selectfont \vspace{0.8pc}
\centerline{\large\bf 
Unification of Rare and Weak Multiple Testing }
\vspace{2pt} 
\centerline{\large\bf Models using Moderate Deviations Analysis }
\vspace{2pt}
\centerline{\large \bf and Log-Chisquared P-values}
\vspace{.4cm} 
\centerline{Alon Kipnis} 
\vspace{.4cm} 
\centerline{\it Reichman University}
 \vspace{.55cm} \fontsize{9}{11.5pt plus.8pt minus.6pt}\selectfont

\newcommand{\Exp}{\mathrm{Exp}}
\newcommand{\lognull}{\mathrm{Exp}(2)}
\newcommand{\LogNormal}{\mathrm{LogNormal}}
\newcommand{\newtext}[1]{{\color{black}#1}}
\newcommand{\Unif}{\mathrm{Unif}}
\newcommand{\Pois}{\mathrm{Pois}}
\newcommand{\FDR}{\mathrm{FDR}}
\newcommand{\Bin}{\mathrm{Bin}}
\newcommand{\Bernoulli}{\mathrm{Bernoulli}}
\newcommand{\Beta}{\mathrm{Beta}}
\newcommand{\Scal}{\mathcal{S}}
\newcommand{\Prp}[1]{\Pr\left[#1 \right]}
\newcommand{\reals}{\mathbb R}
\newcommand{\simiid}{\overset{\mathsf{iid}}{\sim}}
\newcommand{\naturals}{\mathbb N}
\newcommand{\integers}{\mathbb Z}
\newcommand{\DKL}{\mathrm{D}_{\mathsf{KL}}}
\newcommand{\Ncal}{\mathcal{N}}
\newcommand{\dense}{\mathsf{dense}}
\newcommand{\sparse}{\mathsf{sparse}}
\newcommand{\cdense}{\mathsf{Dense}}
\newcommand{\csparse}{\mathsf{Sparse}}

\newcommand{\coloneq}{:=}
\newcommand{\Hell}{\mathrm{H}}
\newcommand{\Bonf}{\mathsf{Bonf}}
\newcommand{\one}{\mathbf{1}}

\newcommand{\One}[1]{\ensuremath{\mathbf{1}_{{#1}}}}

\newcommand{\HC}{\mathrm{HC}}
\newcommand{\TV}{\mathrm{TV}}
\newcommand{\BJ}{\mathrm{BJ}}
\newcommand{\minP}{\mathrm{min-P}}
\newcommand{\normal}{\mathsf{normal}}
\newcommand{\onesample}{\mathsf{one-sample}}
\newcommand{\twosample}{\mathsf{two-sample}}
\newcommand{\Var}[1]{\mathrm{Var}\left[ #1\right]}
\newcommand{\ex}[1]{\ensuremath{\mathbb{E}\left[ #1\right]}}
\newcommand{\exsub}[2]{\ensuremath{\mathbb{E}_{#1}\left[ #2\right]}}

\newcommand{\sigTwoColor}{blue!60!white}
\newcommand{\sigHalfColor}{green!60!black}
\newcommand{\sigOneColor}{red!90!black}
\newcommand{\HCtextcolor}{red!40!white}
\newcommand{\normalTwoSmpColor}{\HCcolor}
\newcommand{\chisqColor}{orange}

\newcommand{\betab}{\boldsymbol\beta}
\newcommand{\SNR}{\mathsf{SNR}}
\newcommand{\deltab}{\boldsymbol\delta}
\newcommand{\pib}{\boldsymbol\pi}
\newcommand{\alphab}{\boldsymbol\alpha}
\newcommand{\phib}{\boldsymbol\phi}
\newcommand{\Phib}{\boldsymbol\Phi}
\newcommand{\rhob}{\boldsymbol\rho}
\newcommand{\Rb}{\bf R}
\newcommand{\gammab}{\boldsymbol\gamma}
\newcommand{\Gammab}{\boldsymbol\Gamma}
\newcommand{\kappab}{\boldsymbol\kappa}
\newcommand{\Id}{{\mathbf{I}}}  
\newcommand{\xb}{\boldsymbol x}
\newcommand{\fb}{\boldsymbol f}
\newcommand{\Bb}{\boldsymbol B}
\newcommand{\mb}{\boldsymbol m}
\newcommand{\Hrm}{\mathrm H}
\newcommand{\db}{\boldsymbol d}
\newcommand{\wb}{\boldsymbol w}
\newcommand{\vb}{\boldsymbol v}
\newcommand{\epsilonb}{\boldsymbol \epsilon}
\newcommand{\yb}{\boldsymbol y}
\newcommand{\ub}{\boldsymbol u}
\newcommand{\hb}{\boldsymbol h}
\newcommand{\zb}{\boldsymbol z}
\newcommand{\nb}{\boldsymbol n}
\newcommand{\cb}{\boldsymbol c}
\newcommand{\eb}{\boldsymbol e}

\newcommand{\MSE}{\mathsf{MSE}}

\newcommand{\hugeskip}{\vspace{2cm}}

\newcommand{\SE}[1]{\mathrm{SE} \left[#1\right]}
\newcommand{\Cov}{\mathrm{Cov}}

\newcommand{\Def}{{\color{blue} Definition}}
\newcommand{\Exm}{{\color{magenta} Example}}
\newcommand{\Exms}{{\color{magenta} Examples}}
\newcommand{\Thm}{{\color{Maroon} Theorem}}
\newcommand{\Property}{{\color{Maroon} Property}}
\newcommand{\Wrn}{{\color{red} Warning}}
\newcommand{\Terminology}{{\color{Maroon} Terminology}}

\newcommand{\Mcal}{\mathcal M}

\newcommand{\nr}[1]{{\sf\textcolor{red}{[#1]}}}

\newcommand{\abs}[1]{\ensuremath{\left\vert#1\right\vert}}
\newcommand\ZuWeis{\mathrel{\mathop:\!\!=}} 
\newcommand\WeisZu{\mathrel{=\!\!\mathop:}}

\newcommand{\N}{\ensuremath{\mathds{N}}} 
\newcommand{\R}{\ensuremath{\mathds{R}}}  
\newcommand{\T}{\ensuremath{\mathds{T}}}  

\newcommand{\Pp}{\ensuremath{\mathsf{\bar{P}}}}  %
\newcommand{\Pf}{\ensuremath{\textbf{f}}}  
\newcommand{\E}{\operatorname{E}} 
\newcommand{\argmin}{\operatorname{argmin}}   
\newcommand{\proj}{\operatorname{proj}}
\newcommand{\rank}{\operatorname{rank}}
\newcommand{\tr}{\operatorname{Tr}}
\newcommand{\diag}{\operatorname{diag}}
\newcommand{\sign}{\operatorname{sign}}


\begin{quotation}
\noindent {\it Abstract:}
Rare and Weak models for multiple hypothesis testing assume that only a small proportion of the tested hypotheses concern non-null effects and the individual effects are only moderately large, so they generally do not stand out individually, for example in a Bonferroni analysis. Such models have been studied in quite a few settings, for example in some cases studies focused on an underlying Gaussian means model for the hypotheses being tested; in some others, Poisson and Binomial. Such seemingly different models have the following common structure. Summarizing the evidence of individual tests by the negative logarithm of its P-value, the model is asymptotically equivalent to a situation in which most negative log P-values have a standard exponential distribution but a small fraction might have an alternative distribution which is approximately noncentral chisquared on one degree of freedom. This log-chisquared approximation is different from the log-normal approximation of Bahadur. The latter is unsuitable for analyzing Rare and Weak multiple-testing models.

We characterize the asymptotic performance of global tests combining asymptotic log-chisquared P-values in terms of the chisquared mixture parameters: the scaling parameter controlling heteroscedasticity, the non-centrality parameter, and the parameter controlling the rarity of individual non-null effects. In a phase space involving the last two parameters, we derive a region where all tests are asymptotically powerless. Outside of this region, the Berk-Jones and the Higher Criticism tests have maximal power. Inference techniques based on the minimal P-value, false-discovery rate controlling, and Fisher's combination test have sub-optimal asymptotic phase diagrams. Our analysis yields the asymptotic power of global testing in various new rare and weak models, including two-sample heteroscedastic normal mixtures and binomial experiments with perturbed probabilities of success.

\vspace{9pt}
\noindent {\it Key words and phrases:}
multiple testing, sparse mixture, heterogeneous mixture,
higher criticism, P-values.
\par

\end{quotation}\par

\def\thefigure{\arabic{figure}}
\def\thetable{\arabic{table}}

\renewcommand{\theequation}{\thesection.\arabic{equation}}

\fontsize{12}{14pt plus.8pt minus .6pt}\selectfont

\ifsupp
\etocsetnexttocdepth{5}
\etocsettocstyle{\subsubsection*{Contents}}{}
\localtableofcontents
\fi

\section{Introduction}
\setcounter{equation}{0}
\label{sec:intro}

\subsection{Motivation}
Consider a multiple hypothesis testing situation, each test involves a different feature of the data where different features are independent. We are interested in testing a global null hypothesis against the following alternative: the non-null effects are concentrated in a small, but unknown, subset of the hypotheses. In the most challenging situation, effects are not only rare but also weak in the sense that the non-null test statistics are unlikely to provide evidence after Bonferroni's correction. Rare and weak multiple hypothesis testing problems
of this nature arise in a wide range
of situations \citep{donoho2015special}. Specific examples include:
\begin{itemize}
    \item \emph{Sparse (rare) signal detection}. We are interested in intercepting a transmission that occupies few frequency bands out of potentially many, while the occupied bands are unknown to us \citep{tandra2008snr, bayer2020look}. The features are periodogram ordinates associated with individual frequency bands. Evidence for the presence of a signal can be gathered by testing each ordinate against the same exponential distribution. 
    \item \emph{Classification}. Classifying images or other high-dimensional signals usually involves hundreds or more features. In a one-versus-all classification setup, we view the typical response of the features under each class as the null hypothesis. Testing against this null amounts to determining whether the tested signal is associated with that class or not. A situation of wide interest is when inter-class discrimination is due to a small proportion of features out of potentially many, and we do not know which ones they are likely to be \citep{donoho2009feature, ingster2009classification,jin2009impossibility}.
    \item \emph{Detecting rare changes between two high-dimensional distributions}. 
    Testing whether two high-dimensional datasets are simply two different realizations of the same data-generating mechanism is a classical problem in statistics, computer science, and information theory \citep{acharya2012competitive,balakrishnan2018hypothesis,DonohoKipnis2020}. This scenario is formulated as a two-sample testing problem; the null hypothesis states that both samples were obtained from the same high-dimensional parent distribution. The alternative hypothesis states that differences between the mechanisms occur in a small and unknown subspace of the parameters.  
\end{itemize}
Applications as above have motivated a significant body of work in rare and weak multiple testing settings throughout the past two decades, providing fruitful insights for signal detection, feature selection, and classification problems in high dimensions \citep{jin2016rare}. Specific examples of rare and weak multiple testing settings include normal mixtures \citep{ingster2012nonparametric,  jin2003detecting,donoho2004higher,abramovich2006adapting}, binomial mixtures \citep{mukherjee2015hypothesis}, 
linear regression model under Gaussian noise \citep{arias2011global, ingster2010detection},
Poisson mixtures \citep{arias2015sparse}, heteroscedastic normal mixtures  \citep{tony2011optimal}, general mixtures \citep{cai2014optimal, arias2017distribution}, mixture of unknown distributions \citep{delaigle2009higher, delaigle2011robustness,arias2017distribution}, and several two-sample settings \citep{DonohoKipnis2020,GaliliKipnisYachini2022}.

\subsection{Contributions}
In this article, we study one rare and weak multiple testing setting that subsumes the vast majority of these previously studied ones. Our setting is not tied to a specific data-generating model. Instead, we model the behavior of a collection of P-values, each P-value summarizes the evidence of one test statistic against the global null. These P-values may be obtained either from one- or two-sample tests and may represent responses over a variety of models. More generally, the advantages of modeling the distribution of the P-values rather than the data are discussed in \citep{lambert1981influence,lambert1982asymptotic,sackrowitz1999p,boos2011p}. 

Recall that a deviation from the mean of a sequence $X_1,X_2,\ldots,$ of standardized and identically and independently distributed random variables is said to be \emph{moderate} if it is of the form $\sqrt{q\log(n)/n}$ for some $q>0$ \citep{RubinSethuraman1965, zeitouni1998large}. For such deviations, Cram\'er's theorem implies  
\begin{align}
    \label{eq:cramer}
\Pr\left[\left|\frac{X_1 + \ldots + X_n}{\sqrt{n}} \right| >  \sqrt{ 2q \log(n)}\right] \sim n^{-q}, \quad \text{uniformly in } q \geq a>0,
\end{align}
provided the moment-generating function exists. Our key insight says that the log-chisquared approximation (see \eqref{eq:lognormal_appx} below) -- and not the log-normal approximation of \citep{bahadur1960stochastic,lambert1982asymptotic} -- 
is accurate for characterizing the asymptotic power of testing in rare and weak models involving departures on the moderate scale. Consequently, our setting unifies all previously studied rare and weak settings in which moderate deviations analysis applies under one setting we denote as the Rare Moderate Departures (RMD) model. This unification provides new characterizations for the power of some global testing procedures in those earlier studied settings and in several new settings as summarized in Table~\ref{table:models}. Additionally, our analysis guides the calibration of parameters of new high-dimensional signal models to experience a phase transition between the rarity and strength of individual effects; see \citep{GaliliKipnisYachini2022} for an example in the context of survival analysis.


\begin{table}
    \begin{center}
    \begin{tabular}{p{6cm}||c|c|l}
             Departures model &
             \multicolumn{2}{|c|}{\thead{Chisquared \\ parameters}}
             & \thead{studied in} \\
             & $\rho$ & $\sigma$ & \\
             \hline 
             \hline 
             Normal means (heteroscedastic)
             & $r$ & $s$ & \citep{tony2011optimal}
              \\ \hline
             Two-sample normal means ~~~~~ (homoscedastic) & $r/2$ & $1$ & \citep{DonohoKipnis2020}
             \\ \hline 
            {\bf  Two-sample normal means (heteroscedastic)} & $r/2$ & $\sqrt{\frac{1+s^2}{2}}$ &  \\ \hline
             Poisson means & $r$ & $1$ & \citep{arias2015sparse}
             \\ \hline
             Two-sample Poisson means & $r/2$ & $1$ & \citep{DonohoKipnis2020}
             \\ \hline
             \bf{
              Binomial success probability} 
               & $r$ & $ s \cdot r$ &
              \\
    \end{tabular}
    \end{center}
    
    \caption{
    Rare and weak multiple testing settings that are carried under our RMD formulation and their noncentral chisquared parameters. Models appearing in bold are new.}
    \label{table:models}
\end{table}

\subsection{Rare Moderately Departed 
Log-chisquared P-values}
Suppose that the $i$-th test statistic yields the P-value $p_i$, $i=1,\ldots,n$.  We further assume that $p_i\sim \Unif(0,1)$ under the global null, corresponding to the case where the model underlying the $i$-th test statistics has a continuous distribution (we relax this assumption later on). Consequently, $-2\log(p_i)\sim \lognull$, where $\lognull$ is the exponential distribution with mean $2$ (or rate $1/2$), also known as the chisquared distribution with two degrees of freedom $\chi^2_2$. Our model proposes the following alternative: Roughly $n\epsilon$ of the P-values depart from their uniform distribution and instead obey
\begin{align}
    \label{eq:chisq_appx}
-2\log(p_i) \overset{D}{\approx} \left( \mu + \sigma Z\right)^2,\qquad Z\sim\Ncal(0,1).
\end{align}
Here $\overset{D}{\approx}$ indicates a specific form of approximation in distribution that we formalize in Section~\ref{sec:APLC} below. Leaving the details of this approximation aside for now, \eqref{eq:chisq_appx} says that $-2\log(p_i)$ is approximately distributed as a scaled noncentral chisquared random variable (RV) over one degree of freedom with noncentrality parameter $\mu$, and scaling parameter $\sigma$. 
We focus on the case where the rarity parameter $\epsilon$ vanishes while the intensity parameter $\mu$ is only moderately large, making our global testing problem challenging; in some cases, impossible. As we shall see, in this regime the non-null effects are not only rare but are also weak in the sense that they generally do not stand out individually in a Bonferroni analysis.

\subsection{Log-chisquared versus Log-normal}
The emergence of the log-chisquared approximation for P-values is somewhat surprising because this approximation is different from the log-normal approximation developed in \citep{bahadur1960stochastic} and \citep{lambert1982asymptotic}. 
In Section~\ref{sec:logchisq_vs_lognorm}, we show that the log-chisquared distribution fits the distribution of the P-values under moderate departures significantly better than the log-normal distribution. Furthermore, the log-normal approximation does not indicate the correct asymptotic performance of tests under rare and weak multiple testing settings. 
To summarize this last point, we establish here that a rare multiple hypothesis testing setting in which the departures are on the moderate scale corresponds to detecting a few noncentral chisquared signals against an exponential background, rather than detecting a few normal signals as one might have proposed in view of the log-normal approximation. 
Potential applications of this improved approximation, beyond the asymptotic power analysis of multiple hypothesis testing we discuss in this paper, include better estimates of the so-called ``reproducibility probability'' of experiments \citep{boos2011p} and empirical Bayes method for identifying discoveries in large-scale inference \citep{efron2001empirical,pounds2003estimating}; we leave these topics as future work. 

We note that the logarithmic scoring scale for P-values goes back to Fisher, who initially suggested it as a method of ranking success in card-guessing games \citep{fisher1924method}. 
For global testing, Fisher proposed the statistic \citep{fisher1992statistical}
\begin{align}
F_n := \sum_{i=1}^n -2\log(p_i)
\label{eq:fisher},
\end{align}
which has a $\chi_{2n}^2$ distribution under the global null. A test based on $F_n$ is known to be effective in the presence of small effects distributed across the bulk of cases, but not effective under relatively rare and somewhat stronger but individually still weak as our model proposes; see a formal statement about the inadequacy of $F_n$ in our setting in Theorem~\ref{thm:Fisher} below. The logarithmic scale for P-values is now standard in genome-wide association studies (GWAS) (e.g., \citep{balding2006tutorial, pearson2008interpret, price2010new, harold2009genome}) and in other areas \citep{li2012volcano, quaino2014volcano, boos2011p, gibson2021role}. Our setting yields an explicit model for testing rare and weak effects in these applications: testing chisquared departures against an exponential background. A similar model arises in detecting the presence of rare and weak sinusoids in white noise based on the periodogram. For this setting, Fisher's periodogram test is based on the largest periodogram ordinate \citep{fisher1929tests} which is analogous to a Bonferroni analysis.

\subsection{Paper Organization}
In Section~\ref{sec:APLC} we define the RMD setting and analyze the asymptotic properties of tests. In 
Section~\ref{sec:models} we explore several rare and weak signal detection problems that conform to the RMD model formulation. In Section~\ref{sec:logchisq_vs_lognorm} we compare our log-chisquared approximation for the distribution of P-values under the alternative hypothesis and the classical log-normal approximation. Additional discussions are provided in Section~\ref{sec:discussion}. All proofs are in the Supplementary Material \citep{KipnisLogChisqSupp}.

\section{Rare Moderate Departures Setting and Analysis
\label{sec:APLC}}
\setcounter{equation}{0}

\subsection{Multiple Testing with Rare Non-null Effects}
The description in the Introduction above depicts the following global hypothesis testing setting involving a sequence of P-values $p_1,\ldots,p_n$.
\begin{align}
\begin{split}
    H_0 & \quad : \quad -2\log(p_i) \sim \lognull ,\quad i=1,\ldots,n, \\
    H_1^{(n)} & \quad : \quad -2\log(p_i) \sim (1-\epsilon)\lognull + \epsilon Q_i^{(n)}, \quad i=1,\ldots,n,
    \label{eq:hyp_log_n}
\end{split}
\end{align}
where $Q_i^{(n)}$ is a probability distribution specifying the non-null behavior of the $i$-th P-value.

We calibrate the rarity parameter $\epsilon$ to $n$ according to
\begin{align}
\label{eq:calibration_eps}
    \epsilon & = \epsilon_n := n^{-\beta}, 
\end{align}
where $\beta\in(0,1)$. This calibration proposes that for an overwhelming majority of the individual tests, the response under the alternative is indistinguishable from the null. 

The expected proportion of the non-null effects (rarity) of most interest under RMD is $n^{-\beta}$ for $\beta \in (1/2,1)$. Indeed, for less rare effects ($\beta < 1/2$), the signal may be viewed as ``dense" in the sense that tests that are powerful against small but frequent departures can also be asymptotically powerful \citep{arias2011global}. 

\subsection{
The Log-Chisquared Approximation}
We compare $Q_i^{(n)}$ to the non-central and scaled chisquared distribution in the right-hand side of \eqref{eq:chisq_appx} with non-centrality parameter $\mu$ calibrated to $n$ as 
\begin{align}
    \label{eq:calibration_mu}
    \mu = \mu_n(\rho) := \sqrt{2 \rho \log(n)}, \qquad \rho > 0,
\end{align}
and with a fixed scaling parameter $\sigma$. 
Specifically, define the moderately perturbed and scaled chisquared distribution
\begin{align*}
    \chi^2(\rho,\sigma) \overset{D}{=} (\mu_n(\rho) + \sigma Z)^2,\qquad Z\sim \Ncal(0,1),
\end{align*}
where $\overset{D}{=}$ indicates equality in distribution. For the sake of formalizing the approximation in \eqref{eq:chisq_appx}, 
we introduce the function 
\begin{align}
    \alpha(q;\rho,\sigma) := \left( \frac{\sqrt{q}-\sqrt{\rho}}{\sigma}\right)^2
    \label{eq:alpha_def},
\end{align}
and note that 
\begin{align*}
    \lim_{n\to \infty}  \frac{-\log \Pr\left[{ \chi^2(\rho,\sigma) \geq 2q\log(n)}\right]}{\log(n)}  = \alpha(q;r,\sigma),\qquad q > \rho.
\end{align*}
A sequence of distributions $\{Q_i^{(n)}\}_{i=1}^n$ with $\Prp{Q_i^{(n)} \geq 2q\log(n)}>0$ for all $i=1,\ldots,n$ is said to be uniformly moderate chisquared if, for every $q>\rho$,
\begin{align}
\label{eq:APLCC_cond}
    \lim_{n\to \infty} \max_{i=1,\ldots,n} \left| \frac{ -\log \Pr\left[{ Q_i^{(n)} \geq 2q\log(n)}\right]}{\log(n)}  -\alpha(q;\rho,\sigma)\right|=0.
\end{align}
 Namely, we require that the moderate tail probability of $Q_i^{(n)}$ is identical to that of the non-central and scaled chisquared $\chi^2(\rho,\sigma)$. Henceforth, we refer to hypothesis testing problems of the form \eqref{eq:hyp_log_n} in which $\{Q_i^{(n)}\}_{i=1}^n$ is uniformly moderate chisquared as Rare Moderate Departures (RMD) model with log-chisquared parameters $(\rho,\sigma)$. Uniformly moderate chisquaredness is weaker than convergence in distribution in the sense that  \eqref{eq:APLCC_cond} holds whenever
\begin{align}
Q_i^{(n)} \overset{D}{=} \left( \mu_n(\rho) + \sigma Z  \right)^2\left(1+o_p(1)\right),\quad \rho>0,\quad n\to\infty,
    \label{eq:APLCC_cond2}
\end{align}
where $o_p(1)$ indicates a sequence of RVs tending to zero in probability uniformly in $i$ as $n \to \infty$. 

\subsection{Asymptotically Uniform P-values}
\label{sec:randomized_pvals}
We now extend our setting \eqref{eq:hyp_log_n} to situations in which $p_1,\ldots,p_n$ are not uniformly distributed under the null. We do so by considering, instead of \eqref{eq:hyp_log_n}, 
\begin{align}
\begin{split}
    H_0^{(n)} & \quad : \quad -2\log(p_i) \sim E_i^{(n)} ,\quad i=1,\ldots,n, \\
    H_1^{(n)} & \quad : \quad -2\log(p_i) \sim (1-\epsilon)E_i^{(n)} + \epsilon Q_i^{(n)}, \quad i=1,\ldots,n,
    \label{eq:hyp_log_n_appx}
\end{split}
\end{align}
where $Q_i^{(n)}$ satisfies \eqref{eq:APLCC_cond} and where the probability distribution $E_i^{(n)}$ converges to $\lognull$ in the sense that
\begin{align}
    \lim_{n\to \infty} \max_{i=1\ldots,n} \left|\frac{-\log  \Prp{E_i^{(n)} \geq 2q\log(n)} }{\log(n)} - q \right|= 0
    \label{eq:exp_asym}
\end{align}
for every fixed $q>0$. 
This extension of the RMD setting is particularly useful when the distribution of the P-values under the null is only super uniform as in some discrete models \citep{westfall1997multiple}, or when we consider asymptotic P-values rather than exact P-values which is common in large-scale inference from multiple tests \citep{efron2012large}. Most of the properties of RMD models we derive in this paper hold under this extended setting. In this sense, condition \eqref{eq:exp_asym} provides a range of deviation from the specification of the null distribution under which current and previous results concerning rare and moderately large effects hold.

\subsection{Strong Moderate 
Log-Chisquared}
Stronger forms of the chisquared and exponential approximations \eqref{eq:APLCC_cond} and \eqref{eq:exp_asym} are needed to establish an information-theoretic limit of global testing under models \eqref{eq:hyp_log_n} and \eqref{eq:hyp_log_n_appx}. For the chisquared approximation, we require 
\begin{align}
    \lim_{n\to \infty} \max_{i=1\ldots,n} \frac{\left|\log\left( \frac{dQ_i^{(n)}}{d \chi^2(\rho,\sigma)} (2q\log(n))\right)\right|}{\log(n)} = 0
    \label{eq:LR_APLC}
\end{align}
for any $q >\rho$. For the exponential approximation, we require 
\begin{align}
    \lim_{n\to \infty} \max_{i=1\ldots,n} \frac{\left|\log\left( \frac{dE_i^{(n)}}{d \lognull} (2q\log(n))\right) \right|}{\log(n)} = 0
    \label{eq:exp_asym_strong}
\end{align}
for any fixed $q>0$. 
The type of equivalence between $Q_i^{(n)}$ and $\chi^2(\rho,\sigma)$ described in \eqref{eq:LR_APLC} is similar to the setting of \citep{cai2014optimal}. Henceforth, we refer to hypothesis testing problems of the form \eqref{eq:hyp_log_n} under the condition \eqref{eq:LR_APLC} and \eqref{eq:exp_asym_strong} as the \emph{strong} RMD. 
In the Supplementary Material \citep{KipnisLogChisqSupp}, we show that
\eqref{eq:LR_APLC} implies \eqref{eq:APLCC_cond} and that \eqref{eq:exp_asym_strong} implies \eqref{eq:exp_asym}. Note that \eqref{eq:LR_APLC} holds whenever the distribution of each $Q_i^{(n)}$ has a density and satisfies \eqref{eq:APLCC_cond2}.

\subsection{Asymptotic Power and Phase Transition
\label{sec:PT}
}
RMD models experience a \emph{phase transition} phenomenon in the following sense. For some choice of the parameters $r$, $\beta$, and $\sigma$, the two hypotheses are completely indistinguishable. In another region, some tests can asymptotically distinguish $H_1^{(n)}$ from $H_0^{(n)}$ with probability tending to one. Formally, for a given sequence of statistics $\{T_n\}_{n=1}^\infty$, we say that $\{T_n\}_{n=1}^\infty$ is \emph{asymptotically powerful} if there exists a sequence of thresholds $\{h_n\}_{n=1}^\infty$ such that
\[
\Pr_{H_0^{(n)}} \left( T_{n} > h_n \right) + \Pr_{H_1^{(n)}} \left( T_{n} \leq h_n \right) \to 0,
\]
as $n$ goes to infinity. In contrast, we say that $\{T_{n}\}_{n=1}^\infty$ is \emph{asymptotically powerless} if 
\[
\Pr_{H_0^{(n)}} \left( T_{n} > h_n \right) + \Pr_{H_1^{(n)}} \left( T_{n} \leq h_n \right) \to 1,
\]
for any sequence $\{h_n\}_{n\in \mathbb N}$. The so-called phase transition curve is the boundary of the region in the parameter space $(\beta,r)$ in which all tests are asymptotically powerless. \par

Our would-be phase transition curve is
\begin{align}
\rho^*(\beta,\sigma) & :=
\begin{cases}
     (2 - \sigma^2)(\beta - 1/2) & \frac{1}{2} < \beta < 1- \frac{\sigma^2}{4}, \quad 0<\sigma^2<2,\\
    \left(1-\sigma\sqrt{1 -\beta }\right)^2 &  1- \frac{\sigma^2}{4} \leq  \beta < 1, \quad 0<\sigma^2<2,\\
     0 & \frac{1}{2} < \beta <  1-\frac{1}{\sigma^2}, \quad \sigma^2 \geq 2,\\
    \left(1-\sigma\sqrt{1 -\beta }\right)^2 &  1-\frac{1}{\sigma^2} \leq  \beta < 1, \quad \sigma^2 \geq 2.
    \end{cases}
    \label{eq:rho}
\end{align}
Note that
\begin{align}
    \rho^*(\beta,\sigma) = \inf_{r \geq 0} \left\{ \max_{q\in[0,1]} \left[ \frac{q}{2} - \alpha(q;r,\sigma) \right]  > \beta - \frac{1}{2} \right\}.
    \label{eq:rho_as_minmax}
\end{align}

\subsection{Information Theoretic Lower Bound}
One side of the phase transition follows from an information-theoretic lower bound. This bound requires the strong RMD formulation of \eqref{eq:LR_APLC} and \eqref{eq:exp_asym_strong}. 

\begin{theorem}\label{thm:powerlessness}
Consider the hypothesis testing problem \eqref{eq:hyp_log_n_appx}. For every $i=1,\ldots,n$, assume that $Q_i^{(n)}$ is absolutely continuous with respect to $E_i^{(n)}$, and let
\begin{align}
    \label{eq:LR}
    L_i^{(n)}(x) := \frac{d Q_i^{(n)}}{d E_i^{(n)}}(x) 
\end{align}
be the likelihood ratio between the mixture components. Denote
\begin{align}
    \label{eq:alpha_star}
\alpha^*(q; \rho, \sigma) := 2\min_{y\in[r,q]} \left\{\alpha(q;r,\rho) - \frac{y}{2} \right\}.
\end{align}
Suppose that there exists $\gamma>0$ such that, for any $q\in(r,1+\gamma)$,
\begin{subequations}
\label{eq:general_cond}
\begin{align}
\label{eq:general_cond_A}
    \lim_{n \to \infty} \max_{i=1,\ldots,n} \frac{-\log \left( \exsub{X\sim E_i^{(n)}}{L_i^{(n)} (X) \one_{\{X > 2q\log(n) \}}} \right) }{\log(n)} \geq \alpha^*(q; \rho,\sigma),
\end{align}
and
\begin{align}
    \lim_{n\to \infty} \max_{i=1,\ldots,n} \frac{-\log \left( \exsub{X\sim Q_i^{(n)}}{L_i^{(n)} (X) \one_{\{X \leq 2q\log(n) \}}
    }\right)}{\log(n)} \ge \alpha^*(q; \rho,\sigma)
    \label{eq:general_cond_B},
\end{align}
\end{subequations}
If $\rho< \rho^*(\beta,\sigma)$, all tests are asymptotically powerless. 
\end{theorem}
Theorem~\ref{thm:powerlessness} implies
\begin{corollary}
\label{cor:powerlessness}
Consider the hypothesis testing problem \eqref{eq:hyp_log_n_appx} under the strong RMD formulation of \eqref{eq:LR_APLC} and \eqref{eq:exp_asym_strong}. If $\rho < \rho^*(\beta,\sigma)$, all tests are asymptotically powerless.
\end{corollary}

Theorem~\ref{thm:powerlessness} provides conditions for the impossibility of discriminating $H_0^{(n)}$ from $H_1^{(n)}$ in \eqref{eq:hyp_log_n_appx} that are more general than those provided in \citep{cai2014optimal} and in other studies when specialized to our setting.

Figure~\ref{fig:phase_diagram} depicts $\rho^*(\beta,\sigma)$ for three choices of $\sigma$. The function $\rho^*(\beta,\sigma)$ was first derived in \citep{tony2011optimal} to describe the detection boundary of rare and weak normal means with heteroscedastic components. Theorems~\ref{thm:powerlessness} extends this result from \citep{tony2011optimal} to general rare and weak multiple testing models obeying the RMD formulation. We discuss several such models in Section~\ref{sec:models} below.

\begin{figure}
    \begin{center}
    \begin{tikzpicture}
    \begin{axis}[
    width=10cm,
    height=7cm,
    legend style={at={(0.2,1)},
      anchor=north west, legend columns=1},
    ylabel={$\rho$ (intensity)},
    xlabel={$\beta$ (rarity)},
    ytick={0,0.5,1,1.5,2},
    yticklabels={0,0.5,1,1.5,2},
    xtick={0.5,0.6,0.7,0.8,0.9,1},
    xticklabels={.5,.6,.7,.8, .9,1},
    legend cell align={left},
    ymin=0,
    xmin=0.5,
    xmax=1,
    ymax=1.1,
    ]

\addplot[domain=0.4:0.5, color=\sigHalfColor, style=thick, samples = 13, mark=+, mark size=1pt] {x};
\addlegendentry{$\sigma^2=1/2$};

\addplot[domain=0.4:0.5, color=\sigOneColor, style=thick, samples = 13] {x};
\addlegendentry{$\sigma^2=1$};

\addplot[domain=0.4:0.5, color=\sigTwoColor, style=thick, samples = 13, mark=*, mark size=1pt] {x};
\addlegendentry{$\sigma^2=2$};

\def\sig{.5};

\addplot[domain=0.5:1-\sig^2/4, color=\sigHalfColor, style= thick, samples = 13, mark=+, mark size=1pt] {(2-\sig^2)*(x-1/2)};

\addplot[domain=1-\sig^2/4:1, samples = 7, color=\sigHalfColor, style= thick, mark=+, mark size=1pt]
    {(1-\sig*(1-x)^0.5)^2};


\def\sig{1};

\addplot[domain=0.5:1-\sig^2/4, color=\sigOneColor, style=thick, samples = 3] {(2-\sig^2)*(x-1/2)};

\addplot[domain=1-\sig^2/4:1, samples = 31, color=\sigOneColor, style= thick]
    {(1-\sig*(1-x)^0.5)^2};

\def\sig{2};

\addplot[domain=0.5:1-1/(\sig^2), color=\sigTwoColor, style= thick, samples = 13, mark=*, mark size=1pt] {0};

\addplot[domain=1- 1/(\sig^2):1, samples = 21, color=\sigTwoColor, style= thick, mark=*, mark size=1pt]
    {(1-\sig*(1-x)^0.5)^2};
    




\node[left] (topl) at (axis cs:0.6,1.1) {};
\node[right] (botr) at (axis cs:0.94,0.05) {};

\end{axis}


\end{tikzpicture}
\end{center}
    
\caption{
Phase Diagram. 
The phase transition curve $\rho^*(\beta,\sigma)$ of \eqref{eq:rho} defines the detection boundary in all Rare Moderate Departure models. For $\rho < \rho^*(\beta,\sigma)$, all tests are asymptotically powerless. For $\rho > \rho^*(\beta,\sigma)$, some tests, including Higher Criticism and Berk-Johns, are asymptotically powerful. 
} 
\label{fig:phase_diagram}
\end{figure}

\subsection{Optimal Tests
\label{sec:optimal_tests}
}
To complete the phase transition characterization of RMD models initiated in Theorem~\ref{thm:powerlessness}, we consider two tests that are asymptotically powerful whenever $\rho > \rho^*(\beta,\sigma)$.

\subsubsection{Higher Criticism Test}
The Higher Criticism (HC) of the P-values $p_1,\ldots,p_n$ is defined as 
\[
\HC_{n}^* := \max_{1 \leq i \leq n \gamma_0} \sqrt{n} \frac{i/n - p_{(i)}}{\sqrt{p_{(i)}\left(1-p_{(i)}\right)}},
\]
where $p_{(i)}$ is the $i$-th order statistic of $p_1,\ldots,p_n$, and $0<\gamma_0 < 1$ is a fixed parameter \citep{donoho2004higher}. The HC test rejects $H_0^{(n)}$ for large values of $\HC_n^*$. 

In order to characterize the asymptotic power of HC under \eqref{eq:hyp_log_n_appx}, we restrict the potential sub-uniformity of $p_1,\ldots,p_n$ under $H_0^{(n)}$ beyond what is permitted by \eqref{eq:exp_asym} by requiring
\begin{align}
\label{eq:HC_cond}
    \max_{i=1,\ldots,n} -\log \Prp{E_i^{(n)} \geq 2 q \log(n)} \leq q \log(n) - \log(1+n^\frac{q-1}{2})
\end{align}
for all $n$ larger than some $n_0 \in \mathbb N$ and for all $q \in (0,1]$. This restriction is not a concern when $p_1,\ldots,p_n$ are P-values since any super-uniform seqeunce of RVs $\{E_i^{(n)}\}$ satisfies \eqref{eq:HC_cond}. 

\begin{theorem}\label{thm:powerfulness}
Consider the hypothesis testing problem \eqref{eq:hyp_log_n_appx} under \eqref{eq:APLCC_cond}, and suppose that $\{E_i^{(n)}\}$ obey \eqref{eq:exp_asym} and \eqref{eq:HC_cond}. Fix $\gamma_0\in(0,1/2)$. If $\rho > \rho^*(\beta,\sigma)$, then $\HC_{n}^*$ is asymptotically powerful. 
\end{theorem}

\subsubsection{Berk-Jones Test}
Define the P-values
\[
\pi_{i} := \Pr\left( \mathrm{Beta}(i, n-i+1) < p_{(i)} \right),\quad i=1,\ldots,n,
\]
where $\mathrm{Beta}(a,b)$ is the Beta distribution with shape parameters $a,b>0$. The Berk-Jones (BJ) test statistic is defined as \citep{berk1979goodness, moscovich2016exact}
\[
  M_n := \min\{M_n^-,M_n^+\}, \qquad M_n^- := \min_i \pi_{i},\qquad M_n^{+} := \min_i (1-\pi_i).
\]
\begin{theorem}\label{thm:BJpowerfulness}
Consider the hypothesis testing problems \eqref{eq:hyp_log_n} under the RMD condition \eqref{eq:APLCC_cond}. If $\rho >\rho^*(\beta,\sigma)$, than $1/M_n$ is asymptotically powerful. 
\end{theorem}


\subsection{Sub-optimal Tests
\label{sec:other_tests}
}

\subsubsection{Bonferroni and false-discovery rate controlling}
Bonferroni and false-discovery rate (FDR) controlling methods are two popular approaches for inference in a multiple testing scenario \citep{efron2012large}. %
For testing against the family $H_0^{(n)}$, Bonferroni type inference uses the minimal P-value $p_{(1)}$ as the test statistics. The Benjamini-Hochberg (BH) FDR controlling procedure with parameter $q \in (0,1)$ selects the smallest $k^*$ P-values, where $k^*$ is the largest integer $k$ satisfying $p_{(k)}\leq q k/n$ \citep{benjamini1995controlling}. A global test based on this procedure rejects $H_0^{(n)}$ at level $\alpha$ if at least one P-value is selected when $q = h(\alpha)$, for some critical value $h(\alpha)<1$ designed to reject $H_0^{(n)}$ with probability at most $\alpha$ under $H_0^{(n)}$. Namely,
\begin{align}
    \label{eq:FDR_test}
\text{Reject $H_0^{(n)}$ if and only if ~~} \min_{1\leq i\leq n} \frac{p_{(i)}}{i/n} \leq h(\alpha).
\end{align}
Note that since the BH procedure controls the family-wise error rate at level $q$ under $H_0^{(n)}$, one can use $h(\alpha) = \alpha$, but our analysis is not restricted to this choice of $h(\alpha)$. 
\par
For a RMD model, both procedures turn out to be asymptotically powerful (respectively, powerless) within the exact same region. The phase transition curve distinguishing powerfulness from powerlessness is given by 
\begin{align}
\rho_{\Bonf}(\beta,\sigma) := \begin{cases}
\left(1-\sigma\sqrt{1 -\beta }\right)^2, & 1/2<\beta<1,\quad \sigma^2 < 2, \\
\left(1-\sigma\sqrt{1 -\beta }\right)^2 &  1-\frac{1}{\sigma^2} \leq \beta < 1, \quad \sigma^2 > 2, \\
0, &  \beta < 1-\frac{1}{\sigma^2}, \quad \sigma^2 > 2.
    \end{cases}
    \label{eq:rho_bonf}
\end{align}
\begin{theorem} \label{thm:Bonferroni}
Consider the hypothesis testing problem \eqref{eq:hyp_log_n_appx} under the RMD conditions \eqref{eq:APLCC_cond} and \eqref{eq:exp_asym}. $T_n^{\Bonf} = -\log(p_{(1)})$ is
 asymptotically powerless whenever $\rho < \rho_\Bonf(\beta,\sigma)$ and asymptotically powerful whenever $\rho > \rho_{\Bonf}(\beta,\sigma)$.
\end{theorem}
\begin{theorem}
\label{thm:FDR}
Consider the hypothesis testing problem \eqref{eq:hyp_log_n_appx} under the RMD conditions \eqref{eq:APLCC_cond} and \eqref{eq:exp_asym}. A test based on \eqref{eq:FDR_test} is asymptotically powerless whenever $r < \rho_\Bonf(\beta,\sigma)$ and asymptotically powerful whenever $r>\rho_{\Bonf}(\beta,\sigma)$.
\end{theorem}
Theorems~\ref{thm:Bonferroni} and \ref{thm:FDR} imply that both Bonferroni and FDR type inference are asymptotically optimal for $\sigma < 2$ only when $\beta<1/2$ or $(4-\sigma^2)/4 < \beta$. This situation is similar to the case of the Gaussian means model studied in \citep{donoho2004higher}, implying that under small variances and moderate rarity the evidence for discriminating $H_0^{(n)}$ from $H_1^{(n)}$ are not amongst sets of the form $\{p_i,\,:\,p_i < qk/n,\,q\in(0,1),\,k=1,\ldots,n\}$. Asymptotically, in this case, optimal discrimination is achieved by considering P-values in the much wider range $\{p_i\,:\, p_i < n^{-(1-\delta)}\}$ for some $\delta>0$. This range is considered by HC and BJ, but not by FDR or Bonferroni. 

\subsubsection{Fisher's Combination Test}
We conclude this section by noting that Fisher's combination test \eqref{eq:fisher} 
is asymptotically powerless for all $\beta>1/2$.  \begin{theorem}
\label{thm:Fisher}
Consider the hypothesis testing problem \eqref{eq:hyp_log_n_appx} under \eqref{eq:APLCC_cond}. $F_n$ of \eqref{eq:fisher} is asymptotically powerless whenever $\beta>1/2$.
\end{theorem}


\section{Examples of Rare Moderate Departures Models \label{sec:models}}
\setcounter{equation}{0}
We consider below various examples of rare and weak multiple testing settings that are carried under our RMD formulation. We summarize those in Table~\ref{table:models}. In most cases, these settings were previously studied, however, without the RMD formulation and without deriving all properties stated in Theorems~\ref{thm:powerlessness}-\ref{thm:Fisher}. We indicate these earlier studies at the end of each example. 


\subsection{Heteroscedastic Normal Mixture
\label{subsec:normal_means}
}

Consider testing the presence of a rare location and variance departure in a Gaussian model as in
\begin{align}
\begin{split}
    H_0 & ~~ : ~~ X_i \sim \Ncal(0,1),\qquad i=1,\ldots,n, \\
    H_1 & ~~ : ~~ X_i \sim (1-\epsilon) \Ncal(0,1) + \epsilon \Ncal(\mu,s^2),\qquad i=1,\ldots,n,
    \label{eq:hyp_normal}
    \end{split}
\end{align}
with $s>0$. The relation between the model \eqref{eq:hyp_normal} to \eqref{eq:hyp_log_n} is via the test
\begin{align}
\label{eq:pval_norm}
p_i = \bar{\Phi}(X_i),\qquad  \bar{\Phi}(x) := \Pr( \Ncal(0,1) > x),\qquad i=1,\ldots,n.
\end{align}
Standard facts about Mills' ratio (see, e.g.,  \citep{grimmett2020probability}) imply
\begin{align}
    \label{eq:Mills_ratio}
-2\log\left(\bar{\Phi}(x)\right) \sim -2\log \left(\frac{\phi(x)}{|x|}\right) = x^2(1+o(1)),
\end{align}
as $x \to \infty$. Consequently, under $H_1$, the distribution of $- 2\log(p_i)$ is of the form
\[
 ( 1-\epsilon)\lognull +  \epsilon\, Q_i(\mu,s),
\]
where $Q_i(\mu,s)$ is a probability distribution obeying
\begin{align}
Q_i(\mu, s) \overset{D}{=} (s Z+ \mu)^2 (1+o_p(1)),\qquad Z \sim \Ncal(0,1),
\label{eq:APLCC_cond3}
\end{align}
as $\mu \to \infty$. For $\mu = \mu_n(r)=\sqrt{2 r \log(n)}$, the last evaluation implies that $Q_i(\mu,s)$ satisfies \eqref{eq:APLCC_cond}. Since each $Q_i(\mu,s)$ also has a density, the P-values of \eqref{eq:pval_norm} 
correspond to the strong RMD model formulation with log-chisquared parameters $(\rho,\sigma) = (r,s)$. 

Previous studies of the setting \eqref{eq:hyp_normal} were conducted by  \citep{tony2011optimal}, which derived the optimal phase transition curve $\rho^*(\beta,\sigma)$ and showed that it is attained by HC of the P-values \eqref{eq:pval_norm}. The homoscedastic case $s^2=1$ was initially studied by \citep{ingster1996some}, \citep{jin2003detecting}, and \citep{donoho2004higher}.

\subsection{Two-Sample Heteroscedastic Normal Mixture}

A two-sample version of \eqref{eq:hyp_normal} takes the form:
\begin{align}
\begin{split}
    H_0 & \, : \, X_i,Y_i \sim \Ncal(\nu_i,1),\quad i=1,\ldots,n, \\
    H_1 & \, : 
    \begin{cases}
    X_i \sim \Ncal(\nu_i,1), \\
    Y_i \sim (1-\epsilon) \Ncal(\nu_i,1) + \epsilon \Ncal(\nu_i+\mu,s^2)
    \end{cases},\quad 
     i=1,\ldots,n,
    \label{eq:hyp_normal_twosample}
    \end{split}
\end{align}
where $\nu_1,\ldots,\nu_n$ is a sequence of \emph{unknown} means. For this setting, consider the P-values
\begin{align}
    \label{eq:p_vals_two_sample}
    p_i := \bar{\Phi}\left( \frac{Y_i-X_i}{\sqrt{2}} \right).
\end{align}
Notice that, with $\tilde{Y}_i \sim \Ncal(\nu_i+\mu,s^2)$ and $X_i \sim \Ncal(\nu_i,1)$, Mills' ratio \eqref{eq:Mills_ratio} implies
\[
-2 \log\left( \bar{\Phi}\left( \frac{\tilde{Y}_i-X_i}{\sqrt{2}}  \right) \right) \overset{D}{=} \left( \sqrt{\frac{1+s^2}{2}}Z + \frac{\mu}{\sqrt{2}} \right)^2 (1+o_p(1)),\qquad Z \sim \Ncal(0,1),
\]
as $\mu \to \infty$. Therefore, under $H_1$, we have that the distribution of $-2\log(p_i)$ is of the form
\begin{align}
    \label{eq:chisquared_twosample}
(1-\epsilon)\lognull + \epsilon 
 Q_i(\mu,s), 
\end{align}
where $Q_i(\mu,s)$ is a probability distribution obeying
\[
Q_i(\mu,s) \overset{D}{=} \left( \sqrt{\frac{1+s^2}{2}}Z + \frac{\mu}{\sqrt{2}} \right)^2 (1+o_p(1)),\qquad Z \sim \Ncal(0,1),
\]
as $\mu\to \infty$. It follows that with $\mu$ calibrated to $n$ as in \eqref{eq:calibration_mu}, $Q_i(\mu,s)$ satisfies 
\eqref{eq:APLCC_cond} with mean parameter $\mu_n'(r) = \mu_n(r)/\sqrt{2} = \sqrt{r \log(n)}$ and scaling parameter $\sqrt{(1+s^2)/2}$, hence the P-values \eqref{eq:p_vals_two_sample} corresponds to the strong RMD model formulation with log-chisquared parameters $\rho = r/2$ and $\sigma = \sqrt{(1+s^2)/2}$. 

In order to derive a phase transition curve for this model, we start from \eqref{eq:rho}, adjusting for the scaling factor $2$ in the non-centrality parameter compared to \eqref{eq:calibration_mu}, and substituting $\sqrt{(1+s^2)/2}$ for the standard deviation. We obtain:
\begin{align}
\label{eq:rho_twosample}
\rho_{\twosample}^*(\beta,s) &:=
\begin{cases}
     (3 - s^2)(\beta - 1/2) & \frac{1}{2} < \beta <  \frac{7-s^2}{8}, \quad 0<s^2<3, \\
    2\left(1-\frac{1+s^2}{2}\sqrt{1 -\beta }\right)^2 &  \frac{7-s^2}{8} \leq  \beta < 1, \quad 0<s^2<3, \\
     0 & \frac{1}{2} < \beta <  \frac{s^2-1}{s^2+1}, \quad s^2 \geq 3,\\
    2\left(1-\frac{1+s^2}{2}\sqrt{1 -\beta }\right)^2 &  \frac{s^2-1}{s^2+1} \leq  \beta < 1, \quad s^2 \geq 3.
    \end{cases}
\end{align}
Figure~\ref{fig:phase_diagram_twosample} depicts $\rho_{\twosample}^*(\beta,s)$ for several values of $s$ and compare it with $2\rho^*(\beta,s)$. 

To the best of our knowledge, the curve $\rho_{\twosample}^*(\beta,s)$ is new; the case $s=1$ was considered in \citep{DonohoKipnis2020}. 

\begin{figure}
    \begin{center}
    \begin{tikzpicture}
    \begin{axis}[
    width=10cm,
    height=7cm,
    legend style={at={(0.2,1)},
      anchor=north west, legend columns=1},
    ylabel={$r$ (intensity)},
    xlabel={$\beta$ (rarity)},
    ytick={0,0.5,1,1.5,2},
    yticklabels={0,0.5,1,1.5,2},
    xtick={0.5,0.6,0.7,0.8,0.9,1},
    xticklabels={.5,.6,.7,.8, .9,1},
    legend cell align={left},
    ymin=0,
    xmin=0.5,
    xmax=1,
    ymax=2.1,
    ]

\addplot[domain=0.4:0.5, color=\sigHalfColor, style=thick, samples = 13, mark=+, mark size=2pt] {x};
\addlegendentry{$s^2=1/2$};

\addplot[domain=0.4:0.5, color=\sigOneColor, style=ultra thick, samples = 13] {x};
\addlegendentry{$s^2=1$};

\addplot[domain=0.4:0.5, color=\sigTwoColor, style=thick, samples = 13, mark=*, mark size=1pt] {x};
\addlegendentry{$s^2=2$};

\def\sigt{.5};
\def\sig{((1+\sigt^2)/2)^0.5};

\addplot[domain=0.5:1-\sig^2/4, color=\sigHalfColor, style=thick, samples = 17, mark=+, mark size=2pt] {2*(2-\sig^2)*(x-1/2)};

\addplot[domain=1-\sig^2/4:1, samples = 11, color=\sigHalfColor, style=thick,mark=+, mark size=2pt]
    {2*(1-\sig*(1-x)^0.5)^2};

\def\sigt{1};
\def\sig{((1+\sigt^2)/2)^0.5};

\addplot[domain=0.5:1-\sig^2/4, color=\sigOneColor, style=thick, samples = 3] {2*(2-\sig^2)*(x-1/2)};

\addplot[domain=1-\sig^2/4:1, samples = 37, color=\sigOneColor, style=thick]
    {2*(1-\sig*(1-x)^0.5)^2};

\def\sigt{2};
\def\sig{((1+\sigt^2)/2)^0.5};

\addplot[domain=0.5:1-1/(\sig^2), color=\sigTwoColor, style=thick, samples = 5,mark=*, mark size=1pt] {0};

\addplot[domain=1- 1/(\sig^2):1, samples = 31, color=\sigTwoColor, style=thick,mark=*, mark size=1pt]
    {2*(1-\sig*(1-x)^0.5)^2};

\def\sig{2};


\addplot[domain=1- 1/(\sig^2):1, samples = 21, color=\sigTwoColor, style= thick, opacity=0.2, mark=*, mark size=1pt]
{2*(1-\sig*(1-x)^0.5)^2};

\def\sig{.5};

\addplot[domain=0.5:1-\sig^2/4, color=\sigHalfColor, style= thick, samples = 13, mark=+, mark size=1pt, opacity=.25] {2*(2-\sig^2)*(x-1/2)};

\addplot[domain=1-\sig^2/4:1, samples = 7, color=\sigHalfColor, style= thick, mark=+, mark size=1pt, opacity=.25]
    {2*(1-\sig*(1-x)^0.5)^2};






\node[left] (topl) at (axis cs:0.6,1.1) {};
\node[right] (botr) at (axis cs:0.94,0.05) {};

\end{axis}


\end{tikzpicture}
\end{center}
    
\caption{ 
Two-Sample Phase Diagram. The phase transition curve $\rho_{\twosample}^*(\beta,s)$ of \eqref{eq:rho_twosample} defines the detection boundary for an asymptotically log-chisquared perturbation model \eqref{eq:hyp_log_n}. For $r< \rho_{\twosample}^*(\beta,s)$, all tests are powerless. For $r > \rho_{\twosample}^*(\beta,s)$, the Higher Criticism and the Berk-Jones tests are asymptotically powerful. The faint lines correspond to $2\rho^*(\beta,s)$, where we have $\rho_{\twosample}^*(\beta,1)=2\rho^*(\beta,1)$. 
} 
\label{fig:phase_diagram_twosample}
\end{figure}

\subsection{Poisson Means}
Consider the hypothesis testing problem
\begin{align}
\begin{split}
    H_0 & \quad : \quad X_i \simiid \Pois(\lambda_i),\quad i=1,\ldots,n, \\
    H_1 & \quad : \quad X_i \simiid (1-\epsilon) \Pois(\lambda_i)+ \epsilon \Pois(\lambda_i'),\quad i=1,\ldots,n,
    \label{eq:hyp_Poiss}
    \end{split}
\end{align}
where $\lambda_1,\ldots,\lambda_n$ is a sequence of \emph{known} means and where each $\lambda_i'$ is obtained by perturbing $\lambda_i$ upwards. For this model, we have the P-values
\begin{align}
    \label{eq:pval_Possion}
p_i = \Pp(X_i;\lambda_i),\quad i=1,\ldots,n,
\end{align}
where $\Pp(x;\lambda_i) := \Prp{ \Pois(\lambda_i) \geq x}$. We suppose that the Poisson rates increase with $n$ such that
\begin{equation}
(\min \lambda_i)/\log(n) \to \infty, 
   \label{eq:lambda_cond}
\end{equation}
and the perturbed means are given by 
\begin{align}
    \label{eq:lambda_prime}
\lambda'_i = \lambda_i + \mu_n(r) \sqrt{\lambda_i}
,\qquad i=1,\ldots,n.
\end{align}

Noting that $\log(n)/\lambda'_i \to 0$ and $\lambda'_i-\lambda_i \to \infty$, the behavior of $p_i$ under $H_1^{(n)}$ is obtained using a moderate deviation estimate of the RVs $\Upsilon_{\lambda'_i} \sim \Pois(\lambda'_i)$. This is provided by the following proposition. 
\begin{prop}
\label{prop:poisson_pvals}
Suppose that $\lambda_i$ and $\lambda_i'$ satisfy $\eqref{eq:lambda_cond}$ and \eqref{eq:lambda_prime}. Let $X_i \simiid \Pois(\lambda_i')$ and $S_i = -2\log \Pp(X_i;\lambda_i)$. 
 For every $q > \rho  \geq 0$,
    \begin{align}
\lim_{n \to \infty} \max_{1\leq i \leq n} \left| \frac{-\log \Prp{ S_i \geq 2q \log(n)} }{\log(n)} - \alpha(q;\rho,1) \right|.
\label{eq:APC_Poisson}
    \end{align}
\end{prop}
We conclude that under $H_1$, \eqref{eq:pval_Possion} -\eqref{eq:lambda_prime}, 
\[
p_i \sim (1-\epsilon) E_i^{(n)} + \epsilon Q_i^{(n)}, \quad i=1,\ldots,n
\]
where $\{Q_i^{(n)}\}_{i=1}^n$ obey \eqref{eq:APLCC_cond} and $\{E_i^{(n)}\}_{i=1}^n$ obey \eqref{eq:exp_asym}. Consequently, the model of \eqref{eq:hyp_Poiss} with P-values \eqref{eq:pval_Possion} is RMD with log-chisquared parameters $\rho=r$ and $\sigma=1$. 


\citep{arias2015sparse} studied the Poisson Rates model \eqref{eq:hyp_Poiss}. They derived the optimal phase transition $\rho^*(\beta,1)$, the Bonferroni phase transition $\rho_{\Bonf}(\beta,1)$, and showed that a version of HC is asymptotically powerful whenever $r >\rho^*(\beta,1)$. 

\subsection{Two-Sample Poisson Means}
\label{sec:two_sample_Poisson}

A two-sample version of \eqref{eq:hyp_Poiss} is given as:
\begin{align}
    \begin{split}
    H_0 &~~ : ~~ X_i,Y_i \simiid \Pois(\lambda_i),\quad i=1,\ldots,n.\\
    H_1 &~~ : ~~ \begin{cases}
    X_i \simiid \Pois(\lambda_i) \\
    Y_i \simiid (1-\epsilon)\Pois(\lambda_i) + \epsilon \Pois(\lambda_i')
    \end{cases},
    \quad i=1,\ldots,n.
        \label{eq:hyp_Poisson_twosample}
    \end{split}
\end{align}
Here $\lambda_1,\ldots,\lambda_n$ is a sequence of \emph{unknown} Poisson rates that satisfy \eqref{eq:lambda_cond}, while $\lambda'_1,\ldots,\lambda'_n$ are defined as in \eqref{eq:lambda_prime}. We summarize the significance of the pair $(X_i, Y_i)$ associated with the $i$-th coordinate by the RV:
\begin{align}
    \label{eq:pvals_Poisson_twosided}
p_i := \bar{\Phi}\left(\sqrt{2 Y_i} - \sqrt{2 X_i}\right).
\end{align}
In order to analyze the behavior of $p_1,\ldots,p_n$ under $H_0^{(n)}$ and $H_1^{(n)}$, note that the transformed Poisson RV $\sqrt{X}$, $X\sim \Pois(\lambda)$, is variance stable:
\[
2\sqrt{X} - 2\sqrt{\lambda} \to \Ncal(0,1).
\]
Under $H_1$,  \eqref{eq:lambda_cond} and \eqref{eq:lambda_prime} imply
\begin{align}
    \label{eq:lambda_hel}
\sqrt{\lambda_i'}(1+o(1)) = \sqrt{\lambda_i} + \mu_n(r)/2,
\end{align}
where $o(1)\to 0$ as $n \to \infty$ uniformly in $i$. Consequently, with $\Upsilon_{\lambda'_i} \sim \Pois(\lambda'_i)$,
\begin{align*}
& \sqrt{2\Upsilon_{\lambda'_i}}-\sqrt{2 X_i} = \sqrt{2\Upsilon_{\lambda'_i}} - \sqrt{2 \lambda'_i}-\left(\sqrt{2 X_i} -  \sqrt{2\lambda_i} \right) + \left(\sqrt{2 \lambda'_i} - \sqrt{2 \lambda_i} \right) \\
 & \overset{D}{=} \left(Z + \mu_n(r)/\sqrt{2}\right)(1+o_p(1)),\qquad Z\sim \Ncal(0,1),
\end{align*}
as $n\to \infty$. By setting
\[
\pi_i := \bar{\Phi}\left(\sqrt{2\Upsilon_{\lambda'_i}} - \sqrt{2 X_i} \right), 
\]
combining Mill's ratio \eqref{eq:Mills_ratio} and \eqref{eq:lambda_hel}, we obtain
\begin{align}
    \label{eq:two_sample_Poisson_H1}
-2 \log (\pi_i) & \overset{D}{=} (Z+ \mu_n(r)/\sqrt{2})^2(1+o_p(1))  \\
& = (Z+ \sqrt{r \log(n)})^2(1+o_p(1)).
\nonumber
\end{align}
The last evaluation suggests that \eqref{eq:APLCC_cond2} holds with log-chisquared parameters $(\rho,\sigma) = (r/2,1)$. 
The full argument follows from the proposition below. 
\begin{prop}
    \label{prop:twosample_Poisson}
    Suppose that $\lambda_1,\ldots,\lambda_n$ satisfy
    \[
\min \lambda_i / \log(n) \to \infty \quad \text{as} \quad n \to \infty.
    \]
    Set $\lambda_i'(r) := \lambda_i + \sqrt{r \lambda_i \log(n)}$ for some $r \geq 0$. Assume that $X \sim \Pois(\lambda_i)$, $Y \sim \Pois(\lambda'_i)$, and let 
    \[
\pi_i := \bar{\Phi}\left(\sqrt{2Y_i} - \sqrt{2X_i} \right),
\quad i=1,\ldots,n.
    \]
    Then 
    \begin{align}
\lim_{n \to \infty} \max_{1\leq i \leq n} \left| \frac{-\log \Prp{ -2\log(\pi_i) \geq 2q \log(n)} }{\log(n)} - \alpha(q;r/2,1) \right| = 0 . \nonumber
    \end{align}
\end{prop}

\citep{DonohoKipnis2020} studied a two-sided perturbation model similar to \eqref{eq:hyp_Poisson_twosample} and proposed to use P-values of an exact binomial test as in
\begin{align} \label{eq:pvals_binomial_allocation}
p'_i := \Prp{ \left|\Bin(X_i+Y_i,1/2) - \frac{X_i+Y_i}{2}\right| \leq \left|\frac{X_i-Y_i}{2}\right| },
\end{align}
which have several advantages over \eqref{eq:pvals_Poisson_twosided} in practice. Our RMD formulation shows that the asymptotic properties of the tests in Section~\ref{sec:APLC} based on either collection of P-values under \eqref{eq:hyp_normal_twosample} are identical, e.g., the optimal phase transition is given by $\rho^*_{\twosample}(\beta,1)$.

\subsection{Perturbed Binomial Experiments \label{sec:binomial}}
Suppose that our data consists of $n$ independent samples from a binomial distribution:
\begin{align}
    \label{eq:binomial_data}
X_i \simiid \Bin(m_i,q_i),\quad m_i \in \mathbb N, \quad i = 1,\ldots,n.
\end{align}
We are interested in testing the null hypothesis $H_0\,:\,q_1= \ldots =q_n=1/2$ against an alternative in which we have $q_i > 1/2$ for a small fraction of the indices. It is natural to use P-values from the exact binomial test
\begin{align}
\label{eq:pvals_binomial}
p_i &:= p_{\mathsf{Bin}}(X_i) := \Prp{ \Bin(m_i,1/2) \geq X_i},
\end{align}
although other options are available, e.g. testing for overdispersion \citep{dean1992testing}. The alternative hypothesis is specified as
\[
H_1^{(n)} \,:\, X_i \simiid (1-\epsilon) \Bin(m_i, 1/2) + \epsilon\Bin(m_i, 1/2 + \delta)  ,\quad i=1,\ldots,n, 
\]
for small $\delta$ and $\epsilon$. 

For $X_i \sim \Bin(m_i, q_i)$, the normal approximation 
\[
X_i  \approx \Ncal\left(m_i q_i, m_i q_i(1-q_i) \right).
\]
suggests that deviations on the moderate scale arise by the calibration
\begin{align}
\label{eq:binomial_calibration}
    m_i & = \frac{2\log(n)}{s}(1+o(1))\quad \text{and}\quad 
    \delta = \sqrt{s \cdot r/4},
\end{align}
where $s>0$ and $r \geq 0$ are parameters satisfying $r\cdot s < 1$. The proposition below implies that under such calibration and with $\epsilon = n^{-\beta}$ for $\beta \in (1/2,1)$,  \eqref{eq:binomial_data} with the P-values \eqref{eq:pvals_binomial} corresponds to RMD model with $\rho=r$ and $\sigma^2 = 1 - s r$. 
\begin{prop}
\label{prop:binomial}
Consider $X \sim \Bin(m, 1/2 + \delta)$ with $\delta$ and $m$ calibrated to $n$ as in \eqref{eq:binomial_calibration}. Then for all $q \geq r$, 
\begin{align}
\label{eq:prop:binomial}
    \lim_{n \to \infty} \max_{1 \leq i \leq n} \left| \frac{- \log \Prp{-2\log \left(p_{\mathsf{Bin}}(X)\right)} \geq 2 q \log(n)}{\log(n)} - \alpha\left(q; r, \sqrt{1 - s \cdot r} \right) \right| = 0.
\end{align}
\end{prop}
The optimal phase transition of the RMD model corresponding to \eqref{eq:binomial_data} under \eqref{eq:pvals_binomial} is defined by triplets $(s, r, \beta)$ satisfying 
\begin{align*}
    r = \begin{cases}
        (1 + s \cdot r) \left( \beta - 1/2 \right) & \frac{1}{2} < \beta < \frac{3+s \cdot r}{4}, \\
        \left(1 - \sqrt{(1-s \cdot r)(1-\beta)}  \right)^2 &  \frac{3+s \cdot r}{4} \leq \beta < 1. 
    \end{cases}
\end{align*}
When $s < 1/(\beta-1/2)$, the last relation defines the curve 
\begin{align}
\label{eq:rho_bin}
    \rho^*_{\mathsf{Bin}}(\beta,s) := \begin{cases}
        \frac{\beta - \frac{1}{2}}{1 - s(\beta-\frac{1}{2})} & \frac{1}{2} < \beta < \frac{3 + \frac{1 - \sqrt{1-s}}{1+ \sqrt{1-s}}}{4}, \quad s \leq 1,\\
        \left( \frac{1 - \sqrt{(1-\beta)(1-s\beta)} }{1 + s(1-\beta)} \right)^2 & 
\frac{3 + \frac{1 - \sqrt{1-s}}{1+ \sqrt{1-s}}}{4}
         \leq \beta < 1,\quad  s \leq 1,\\
         \frac{\beta - \frac{1}{2}}{1 - s(\beta-\frac{1}{2})} & \frac{1}{2} < \beta < \frac{1}{2} + \frac{1}{s}, \quad  s \geq 1,\\
         \infty & \frac{1}{2} + \frac{1}{s} \leq \beta, \quad  s \geq 1; 
    \end{cases}
\end{align}
see an illustration in Figure~\ref{fig:phase_diagram_binomial}.  Consequently, for $r < \rho^*_{\mathsf{Bin}}(\beta,s)$ or $s>1/(\beta - 1/2)$, all tests are asymptotically powerless while some tests are asymptotically powerful when $s \leq 1$ and $r > \rho^*_{\mathsf{Bin}}(\beta,s)$. 

To the best of our knowledge, the curve $\rho_{\mathsf{Bin}}^*(\beta,s)$ is new. \citep{mukherjee2015hypothesis} studied the model \eqref{eq:binomial_data} in the context of sparse binary regression under a coarser calibration that corresponds to the limit $s \to 0$ in \eqref{eq:binomial_calibration}. In this case, $\rho^*_{\mathsf{Bin}}(\beta,s)$ converges to $\rho^*(\beta,1)$, in accordance with the results of \citep{mukherjee2015hypothesis}. 

\begin{figure}
    \begin{center}
    \begin{tikzpicture}
    \begin{axis}[
    width=10cm,
    height=7cm,
    legend style={at={(0.2,1)},
      anchor=north west, legend columns=1},
    ylabel={$\rho$ (intensity)},
    xlabel={$\beta$ (rarity)},
    ytick={0,0.5,1,1.5,2},
    yticklabels={0,0.5,1,1.5,2},
    xtick={0.5,0.6,0.7,0.8,0.9,1},
    xticklabels={.5,.6,.7,.8, .9,1},
    legend cell align={left},
    ymin=0,
    xmin=0.5,
    xmax=1,
    ymax=2,
    ]

\addplot[domain=0.4:0.45, color=\sigHalfColor, style=thick, samples = 2, mark=-, mark size=1pt] {x};
\addlegendentry{$s=9/4$};

\addplot[domain=0.4:0.45, color=\sigTwoColor, style=thick, samples = 2, mark=+, mark size=1pt] {x};
\addlegendentry{$s=5/4$};

\addplot[domain=0.4:0.45, color=\sigOneColor, style=thick, samples = 2, mark=*, mark size=1pt] {x};
\addlegendentry{$s=1/2$};

\addplot[domain=0.4:0.45, color=black, style=thick, samples = 2, opacity=.4] {x};
\addlegendentry{$\rho^*(\beta,1)$};


\def\sig{1};

\addplot[domain=0.5:1-\sig^2/4, color=black, style=thick, samples = 3, opacity=.4] {(2-\sig^2)*(x-1/2)};

\addplot[domain=1-\sig^2/4:1, samples = 31, color=black, style= thick, opacity=.25]
    {(1-\sig*(1-x)^0.5)^2};

\def\sig{.5};
\def\betazero{(3 + (1-(1-\sig)^0.5)/(1+(1-\sig)^0.5))/4 };

\addplot[domain=0.5: \betazero, color=\sigOneColor, style= thick, samples = 13, mark=*, mark size=1pt] {(x - 1/2)/(1 - \sig*(x - 1/2)};

\addplot[domain=\betazero:1, samples = 17, color=\sigOneColor, style= thick, mark=*, mark size=1pt]
    {( (1 - ((1-\sig*x)*(1-x))^0.5)/(1 + \sig*(1-x) ))^2 };

\def\sig{1.25};

\addplot[domain=0.5: 1, color=\sigTwoColor, style= thick, samples = 17, mark=+, mark size=1pt] {(x - 1/2)/(1 - \sig*(x - 1/2)};

\def\sig{2.25};

\addplot[domain=0.5:0.94, color=\sigHalfColor, style= thick, samples = 17, mark=-, mark size=1pt] {(x - 1/2)/(1 - \sig*(x - 1/2)};

\addplot[color=\sigHalfColor, style= dashed, mark size=1pt]coordinates{(0.944,0) (0.944,2)};

\node[left] (topl) at (axis cs:0.6,1.1) {};
\node[right] (botr) at (axis cs:0.94,0.05) {};

\end{axis}

\end{tikzpicture}
\end{center}
    
\caption{
Phase transitions of multiple binomial experiments with perturbed success probabilities of \eqref{eq:binomial_data}. The phase transition curve $\rho_{\mathsf{Bin}}^*(\beta,s)$ of \eqref{eq:rho_bin} defines the detection boundary in the multiple binomials model of \eqref{eq:binomial_data} under the calibration \eqref{eq:binomial_calibration}. The parameter $s$ controls the number of experiments in individual binomial trials according to \eqref{eq:binomial_calibration} (larger $s$ means fewer trials). $\rho_{\mathsf{Bin}}^*(\beta,s=9/4)$ asymptotes to the dashed line. The case $s \to 0$ corresponds to the homoscedastic case analyzed in \citep{mukherjee2015hypothesis}. 
} 
\label{fig:phase_diagram_binomial}
\end{figure}

\section{Log-Chisquared versus Log-Normal
\label{sec:logchisq_vs_lognorm}}
\setcounter{equation}{0}

\subsection{The Log-normal Approximation}

The log-normal approximation of a P-value under the alternative hypothesis is a tool developed by Bahadur to study the interplay among the test's size, power, and the ``cost'' of attaining new data which is most commonly associated with the sample size \citep{bahadur1960stochastic, gleser1964measure, lambert1982asymptotic}. Informally, suppose that the alternative hypothesis is characterized by a parameter $\theta$ and that $a_n$ is a sequence tending to infinity with $n$ describing the cost of sampling from the population of interest. Bahadur's log-normal approximation says that, under some conditions including asymptotic normality of the test statistic, a P-value $\pi$ under the alternative $H_1 = H_1(\theta,a_n)$ obeys
\begin{align}
\label{eq:lognormal_conv}
\frac{\log(\pi) + a_n c(\theta)}{\sqrt{a_n}} \overset{D}{\longrightarrow} \Ncal\left(0, \tau^2(\theta) \right),
\end{align}
as $n \to \infty$. In the terminology of \citep{lambert1982asymptotic}, 
$c(\theta)$ is Bahadur's half-slope describing the asymptotic behavior of the test's size, i.e., the rate at which $\pi$ goes to zero. The test's power is determined both by $\tau(\theta)$ and $c(\theta)$. It is convenient to write \eqref{eq:lognormal_conv} as 
\begin{align}
-2\log(\pi) \overset{D}{\approx}  \Ncal\left(a_n c(\theta), a_n\tau^2(\theta) \right).
\label{eq:lognormal_appx}
\end{align}

In the sections below, we compare our log-chisquared approximation to the log-normal approximation of \eqref{eq:lognormal_appx} for P-values under the alternative hypothesis.

\subsection{Formal Comparison}
It is well-recognized that \eqref{eq:lognormal_appx} is a \emph{large deviation} estimate of the test statistic in the sense that $c(\theta)$ is a transformation of the statistic's rate function whenever this statistic satisfies a large deviation principle \citep{sievers1969probability,  gleser1980large, stanford1980large}. In contrast, in all RMD models of Section~\ref{sec:models} the alternative hypothesis corresponds to a \emph{moderate deviation} of each test statistic from its null \citep{RubinSethuraman1965}. Consequently, the log-normal approximation of \eqref{eq:lognormal_appx} cannot correctly indicate the asymptotic power of tests under the RMD formulation.

We formally show this last point in the homoscedastic RMD normal \newtext{mixture} model:
\begin{align}
\label{eq:hyp_normal_simple}
\begin{split}
    H_0 \, &: \, X_i \simiid \Ncal(0,1),\quad  i=1,\ldots,n, \\
    H_1 \, &: \, X_i \simiid (1-\epsilon)\Ncal(0,1) + \epsilon \Ncal(\mu,1), ,\quad i=1,\ldots,n,
\end{split}
\end{align}
with the P-values $p_i = \bar{\Phi}(X_i)$. This is the model \eqref{eq:hyp_normal} with $s=1$. Under $H_1$, 
\begin{align}
\begin{split}
-2\log p_i \sim (1-\epsilon)\lognull + \epsilon Q_i,
\label{eq:mixture} 
\end{split}
\end{align}
where the probability distribution $Q_i$ is the subject of our approximation. When $\mu$ and $\epsilon$ are calibrated to $n$ as in \eqref{eq:calibration_mu} and \eqref{eq:calibration_eps}, respectively, we propose in this paper the log-chisquared approximation
\begin{align}
Q_i = Q_i^{(n)} \overset{D}{\approx} \left(\mu_n(r) +  Z\right)^2.
\label{eq:LC_appx}
\end{align}
On the other hand, we have
\begin{align*}
(\mu + Z)^2 = (\mu^2+ 2 \mu Z)(1+o_{p}(1)),\quad \mu \to \infty,
\end{align*}
implying the log-normal approximation:
\begin{align}
    \label{eq:LN_appx1}
Q_i^{(n)} \overset{D}{\approx} \Ncal(\mu_n^2(r),4 \mu_n^2(r) ) = \Ncal\left(2 r\log(n),8 r\log(n)  \right).
\end{align}
In particular, $\theta=r$, $a_n = \log(n)$, $c(\theta)=2r$, and $\tau^2(\theta)= 8 r$ in the notation of \eqref{eq:lognormal_appx}. We now compare approximations \eqref{eq:LC_appx} and \eqref{eq:LN_appx1}. Observe that the success of HC and the BJ tests follows from the behavior of 
\[
\Pr( \pi_i <  n^{-q}), \qquad -2\log(\pi_i) \sim Q_i^{(n)}
\]
for $r < q < 1$ as $n\to \infty$; see the proofs of Theorems~\ref{thm:powerfulness} and \ref{thm:BJpowerfulness} in the Supplementary Materials and an extended analysis in \cite{donoho2021impossibility}. With $Q_i^{(n)}$ as in \eqref{eq:LC_appx}, 
\begin{align}
\Pr( \pi_i < n^{-q}) & =
\Pr(-2\log(\pi_i) >  2q \log(n)) \nonumber \\
& \sim \Pr\left( (\mu_n(r)+Z)^2 \geq 2q \log(n) \right) \label{eq:true_appx}  \\
& \sim \Pr\left( Z \geq \sqrt{\log(n)} (\sqrt{2q}-\sqrt{2r}) \right) \nonumber.
\end{align}
A standard evaluation of the behavior of HC under $H_1^{(n)}$ uses \eqref{eq:true_appx} to show that it is asymptotically powerful for $r>\rho(\beta,1)$ \citep{donoho2004higher}. On the other hand, with $Q_i^{(n)}$ as in \eqref{eq:LN_appx1}, 
\begin{align}
\Pr( \pi_i <  n^{-q})  & =
\Pr(-2\log(\pi_i) >  2q \log(n)) \nonumber \\ 
& \sim \Pr\left( \mu_n(r)^2 + 2  \mu_n(r) Z \geq 2q \log(n) \right) \nonumber  \\
& = \Pr\left( Z \geq \sqrt{\log(n)} \frac{q - r}{\sqrt{2r}} \right) \label{eq:wrong_appx}.
\end{align}
Since 
\[
\frac{q-r}{\sqrt{2r}} \geq \sqrt{2q}-\sqrt{2r},\quad q \geq r > 0,
\]
using the log-normal approximation in a formal exercise of would-be power analysis of HC by replacing \eqref{eq:true_appx} with \eqref{eq:wrong_appx}, incorrectly predicts that HC is powerless for some $r > \rho(\beta,1)$. Specifically, the log-normal approximation incorrectly predicts the phase transition curve:
\begin{align*}
\rho^\dagger(\beta,1) = 
\begin{cases}
\frac{2}{3}\left(2\beta-1\right) &  \frac{7}{8} < \beta, \\
3 - 2\beta - 2\sqrt{(2-\beta)(1-\beta)} & \frac{1}{2} < \beta \leq \frac{7}{8}, 
\end{cases}
\end{align*}
which satisfies $\rho^\dagger(\beta,1) > \rho(\beta,1)$ for $\beta \in (1/2,1)$. 

\subsection{Empirical Comparison}
We now provide an empirical comparison between the log-normal and the log-chisquared approximation under moderate departures using Monte Carlo simulations involving fixed effects (i.e., not rare). In each simulation, we sample data $x_1,\ldots,x_n$ independently from $\Ncal(\mu_n(r),1)$ and consider $\pi_i = \bar{\Phi}(x_i)$ as a P-value under $H_0\,:\, X_i\simiid \Ncal(0,1)$, $i=1,\ldots,n$. 
Of course, in this model, we can characterize the distribution of $\pi_i$ analytically, as well as the deviation of this distribution from the log-chisquared and the log-normal distributions, respectively. The purpose of the simulation is to illustrate the better fit of the log-chisquared approximation. Figure~\ref{fig:sim1} illustrates the results of one simulation with $n=1,000$ (top panels) and one simulation with $n=100,000$ (bottom panels), while the departure intensity parameter $r=1$ is fixed in both cases. The panels on the left show the histogram of $\{-2\log(\pi_i)\}_{i=1}^n$ with the density of the normal distribution $\Ncal(\hat{\mu},\hat{\sigma}^2)$ and the density of the noncentral chisquared distribution $\chi^2_2(\hat{\lambda})$, where $\hat{\mu}$ and $\hat{\sigma}^2$ are the standard mean and variance estimates and $\hat{\lambda} = \hat{\mu}-2$ is the non-centrality estimate. The middle and right panels illustrate QQ-plots of the empirical distribution of $\{-2\log(\pi_i)\}_{i=1}^n$ against 
 $\chi^2_2(\hat{\lambda})$ and $\Ncal(\hat{\mu},\hat{\sigma}^2)$, respectively, showing the better fit of the empirical distribution of $\{-2\log(\pi_i)\}_{i=1}^n$ to $\chi^2_2(\hat{\lambda})$.
 
\begin{figure}
    \centering
    \def \muhat {17.22}
\def \sighat {8.15}
\def \nchat {15.22}
\def \maxvalx {60}
\def \maxvalxfig {60}
\def \maxvaly {0.07}
\def \maxvalyfig {0.07}
\def \maxvalq {60}

\pgfplotsset{scaled y ticks=false}

\begin{tikzpicture}[scale = 1]
	\begin{axis}[
    width=6cm,
    height=4.5cm,
    legend style={at={(1,1)},
      anchor=north, legend columns=1},
    xlabel={\scriptsize $-2\log(p_i)$},
    xlabel style={at={(0.5,0)}},
    ylabel={density},
    ylabel style={at={(0, 0.5)}},
    xtick={0,\maxvalx},
    xticklabels={\scriptsize $0$, \scriptsize $\maxvalx$},
    ytick={\maxvaly},
    yticklabels={\scriptsize \maxvaly},
    ymin=0,
    xmin=0,
    xmax=60,
    ymax=\maxvaly,
    title = {\footnotesize \bf Histogram of $-2\log(p_i)$}
    ]

\addplot graphics
		[xmin=0, xmax=\maxvalxfig, ymin=0, ymax=\maxvalyfig,
		includegraphics={scale=.6, trim=44 38 35 42, clip}]
		{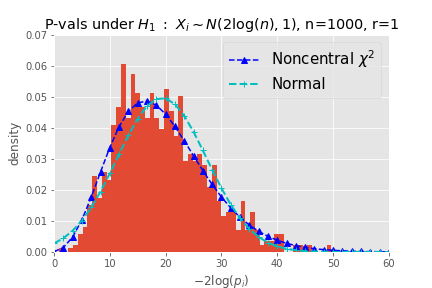};
\end{axis}
\end{tikzpicture}
\begin{tikzpicture}[scale = 1]
	\begin{axis}[
    width=4.5cm,
    height=4.5cm,
    legend style={at={(1,1)},
      anchor=north, legend columns=1},
    ylabel={\footnotesize $Q$-sample},
    xlabel style={at={(0.5,0)}},
    xlabel={\footnotesize $Q$-ChiSq},
    ylabel style={at={(0, 0.5)}},
    xtick={0,\maxvalq},
    xticklabels={\scriptsize $0$, \scriptsize $\maxvalq$},
    ytick={\maxvalq},
    yticklabels={\scriptsize $\maxvalq$},
    ymin=0,
    xmin=0,
    xmax=\maxvalq,
    ymax=\maxvalq,
    title = {\scriptsize \bf Chisquared QQ-Plot }
    ]

\addplot graphics
		[xmin=0, xmax=\maxvalq, ymin=0, ymax=\maxvalq,
		includegraphics={scale=.6, trim=44 38 35 42, clip}]
		{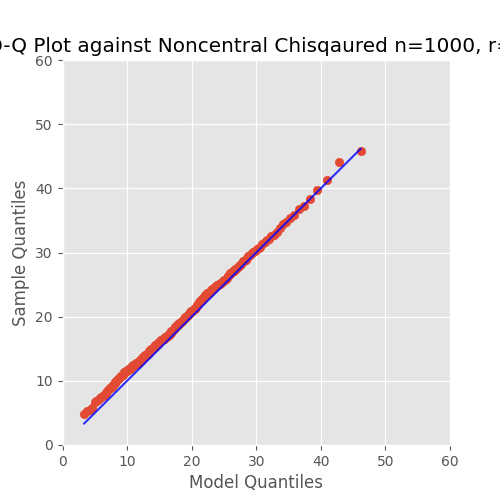};
\end{axis}
\end{tikzpicture}
\begin{tikzpicture}[scale = 1]
	\begin{axis}[
    width=4.5cm,
    height=4.5cm,
    legend style={at={(1,1)},
      anchor=north, legend columns=1},
    ylabel={\footnotesize $Q$-sample},
    xlabel style={at={(0.5,0)}},
    xlabel={\footnotesize $Q$-Normal},
    ylabel style={at={(0, 0.5)}},
    xtick={0,\maxvalq},
    xticklabels={\scriptsize $0$, \scriptsize $\maxvalq$},
    ytick={60},
    yticklabels={\scriptsize $\maxvalq$},
    ymin=0,
    xmin=0,
    xmax=\maxvalq,
    ymax=\maxvalq,
    title = {\footnotesize \bf Normal QQ-Plot}
    ]

\addplot graphics
		[xmin=0, xmax=\maxvalq, ymin=0, ymax=\maxvalq,
		includegraphics={scale=.6, trim=44 38 35 42, clip}]
		{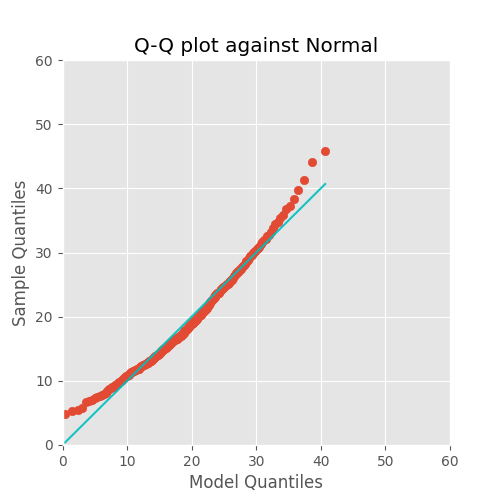};
\end{axis}
\end{tikzpicture}
    
    \hspace{-5pt}\def \muhat {17.22}
\def \sighat {8.15}
\def \nchat {15.22}
\def \maxvalx {60}
\def \maxvalxfig {80}
\def \maxvaly {0.07}
\def \maxvalyfig {0.07}
\def \maxvalq {60}

\pgfplotsset{scaled y ticks=false}

\begin{tikzpicture}[scale = 1]
	\begin{axis}[
    width=6cm,
    height=4.5cm,
    legend style={at={(1,1)},
      anchor=north, legend columns=1},
    xlabel={\footnotesize $-2\log(p_i)$},
    xlabel style={at={(0.5,0)}},
    ylabel={density},
    ylabel style={at={(0, 0.5)}},
    xtick={0,\maxvalx},
    xticklabels={\scriptsize $0$, \scriptsize $\maxvalx$},
    ytick={\maxvaly},
    yticklabels={\scriptsize $\maxvaly$},
    ymin=0,
    xmin=0,
    xmax=\maxvalx,
    ymax=\maxvaly,
    ]

\addplot graphics
		[xmin=0, xmax=\maxvalxfig, ymin=0,ymax=\maxvalyfig,
		includegraphics={scale=.35, trim=55 36 -68 36, clip}]
		{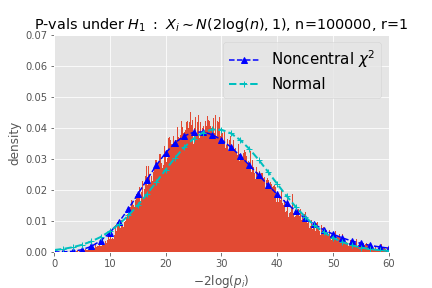};
\end{axis}
\end{tikzpicture}
\begin{tikzpicture}[scale = 1]
	\begin{axis}[
    width=4.5cm,
    height=4.5cm,
    legend style={at={(1,1)},
      anchor=north, legend columns=1},
    ylabel={\footnotesize Q-sample},
    xlabel style={at={(0.5,0)}},
    xlabel={\footnotesize Q-ChiSq},
    ylabel style={at={(0, 0.5)}},
    xtick={0,\maxvalq},
    xticklabels={\scriptsize $0$, \scriptsize $\maxvalq$},
    ytick={\maxvalq},
    yticklabels={\scriptsize $\maxvalq$},
    ymin=0,
    xmin=0,
    xmax=\maxvalq,
    ymax=\maxvalq,
    ]

\addplot graphics
		[xmin=0, xmax=\maxvalq, ymin=0, ymax=\maxvalq,
		includegraphics={scale=.6, trim=44 38 35 42, clip}]
		{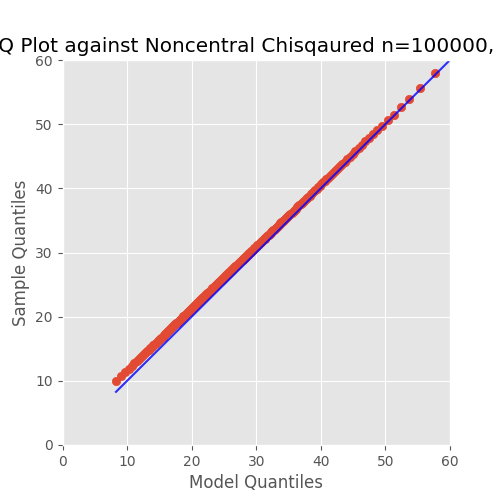};
\end{axis}
\end{tikzpicture}
\begin{tikzpicture}[scale = 1]
	\begin{axis}[
    width=4.5cm,
    height=4.5cm,
    legend style={at={(1,1)},
      anchor=north, legend columns=1},
    ylabel={\footnotesize Q-sample},
    xlabel style={at={(0.5,0)}},
    xlabel={\footnotesize Q-Normal},
    ylabel style={at={(0, 0.5)}},
    xtick={0,\maxvalq},
    xticklabels={\scriptsize $0$, \scriptsize $\maxvalq$},
    ytick={\maxvalq},
    yticklabels={\scriptsize $\maxvalq$},
    ymin=0,
    xmin=0,
    xmax=\maxvalq,
    ymax=\maxvalq,
    ]

\addplot graphics
		[xmin=0, xmax=\maxvalq, ymin=0, ymax=\maxvalq,
		includegraphics={scale=.6, trim=44 38 35 42, clip}]
		{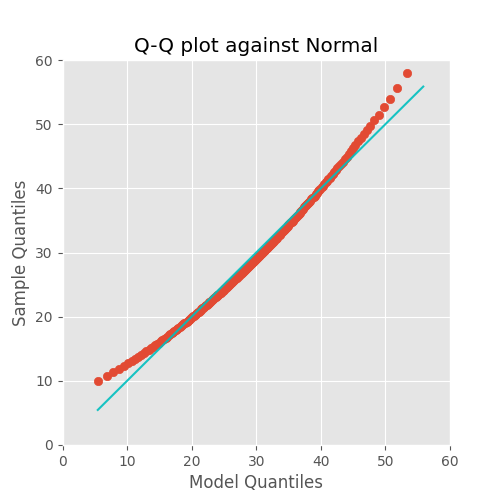};
\end{axis}
\end{tikzpicture}
    
    \caption{
    Comparing log-normal and log-chisquared approximations to moderately perturbed P-values $\pi_1,\ldots,\pi_n$. Here $\pi_i \sim \bar{\Phi}(X_i)$, $X_i=\Ncal(\sqrt{2r\log(n)},1)$, with $n=10^3$ (top) and $n=10^5$ (bottom).  
    Left: histogram of $\{-2\log(\pi_i)\}_{i=1}^n$. Middle: QQ-plots of the empirical distribution of $\{-2\log(\pi_i)\}_{i=1}^n$ against the noncentral chisquared distribution. Right: QQ-plots of the empirical distribution of  $\{-2\log(\pi_i)\}_{i=1}^n$ against the normal distribution.}
    \label{fig:sim1}
\end{figure}

Figure~\ref{fig:sim2} illustrates the results of $1,000$ Monte Carlo simulations with many configurations of $n$ and $r$. For each configuration, we conducted an 
Anderson-Darling (AD) goodness-of-fit test of $\{-2\log(p_i)\}_{i=1}^n$ against
$\chi^2_2(\hat{\lambda})$ and $\Ncal(\hat{\mu},\hat{\sigma}^2)$. The test against the normal (respectively, chisquared) rejects when the AD statistic exceeds its simulated $.95$-th quantile under the null obtained by sampling $10,000$ times from the normal (chisquared) distribution. \newtext{It follows from Figure~\ref{fig:sim2} that the log-chisquared distribution fits the distribution of the P-values under fixed moderate effects much better than the log-normal distribution. The lack of fit of the log-chisquared distribution is only visible when the sample size is large or when the signal is very weak. The analysis in Section~\ref{sec:APLC} implies that this lack of fit is insignificant when the effects are also rare in the sense that the log-chisquared approximation provides the information-theoretic limit of signal detection under RMD. 
}

\begin{figure}
    \begin{center}
    \def \nmax {10000}
\def \rmax {12}
\pgfplotsset{scaled y ticks=false}
\pgfplotsset{scaled x ticks=false}
\begin{tikzpicture}[scale = 1, every mark/.append style={mark size=1pt}]
	\begin{axis}[
    width=7cm,
    height=5.5cm,
    legend style={at={(1,0.45)},
      anchor=north east, legend columns=1},
    xlabel={\footnotesize $r$ (intensity)},
    ylabel={\footnotesize rejection rate},
    xtick={0.1, 5, 10},
    xticklabels={\scriptsize $0.1$, \scriptsize $5$, \scriptsize $10$},
    ytick={0.05, 0.5 ,1},
    yticklabels={\scriptsize $0.05$ ,\scriptsize $0.5$, \scriptsize $1$},
    ymin=0,
    xmin=0.1,
    xmax=\rmax,
    ymax=1.05,
    ]

\addplot table [x=r, y=rej_ad_norm_tab, col sep=comma] {Figs/rej_ad_norm_tab_vs_r.csv};

\addlegendentry{\footnotesize Normal}

\addplot table [x=r, y=rej_ad_chisq_sqrt_tab, col sep=comma] {Figs/rej_ad_chisq_sqrt_tab_vs_r.csv};

\addlegendentry{\footnotesize Noncentral $\chi^2$}

\addplot[dotted, color=black]
coordinates {(0,.05) (\rmax,.05)};

\end{axis}
\end{tikzpicture}
\begin{tikzpicture}[scale = 1, every mark/.append style={mark size=1pt}]
	\begin{axis}[
	xmode=log,
    width=7cm,
    height=5.5cm,
    legend style={at={(1,.45)},
      anchor=north east, legend columns=1},
    xlabel={\footnotesize $n$ (sample size)},
    xticklabels={\scriptsize $10$, \scriptsize $10^2$, \scriptsize $10^3$, \scriptsize $10^4$},
    ytick={.05,.5,1},
   yticklabels={ \scriptsize $0.05$,\scriptsize $0.5$, \scriptsize $1$},
    ymin=0,
    xmin=10,
    xmax=\nmax,
    ymax=1.05,
    ]

\addplot table [x=n, y=rej_ad_norm_tab, col sep=comma] {Figs/rej_ad_norm_tab_vs_n.csv};

\addlegendentry{\footnotesize Normal}

\addplot table [x=n, y=rej_ad_chisq_sqrt_tab, col sep=comma] {Figs/rej_ad_chisq_sqrt_tab_vs_n.csv};

\addlegendentry{\footnotesize Noncentral $\chi^2$}

\addplot[dotted, color=black] coordinates {(10, 0.05) (\nmax, 0.05)};

\end{axis}
\end{tikzpicture}
    \end{center}
    \caption{ Comparing the fit of the empirical distribution of moderately perturbed P-values to the normal and noncentral chisquared distributions. Both panels show the rejection rate of the Anderson-Darling (AD) Goodness-of-fit test at a significance level of $0.05$ (a smaller rejection rate indicates a better fit).
    Left: rejection rate versus perturbation intensity parameter $r$; $n=1,000$ is fixed. Right: rejection rate versus sample size $n$; $r=2$ is fixed. 
}    
    \label{fig:sim2}
\end{figure}



\section{Additional Discussion} \label{sec:discussion}
\setcounter{equation}{0}

\subsection{Heteroscedasticity in RMD models}
\label{sec:heteroscedastic}
The phase transition described by $\rho(\beta,\sigma)$ can be seen as the result of two phenomena: (i) location shift controlled by $\rho$, and (ii) heteroscedasticity controlled by $\sigma^2$. Roughly speaking, increasing the effect of either (i) or (ii) eases detection and reduces the phase transition curve, as seen in Figure~\ref{fig:phase_diagram}. We refer to \citep{tony2011optimal} for a more comprehensive discussion on the effect of (ii) on the phase transition curves. With obvious changes, this discussion is also relevant to the curves $\rho_{\twosample}(\beta;\sigma)$ of \eqref{eq:rho_twosample} and $\rho^*_{\mathsf{Bin}}(\beta,s)$ of \eqref{eq:rho_bin} with $s>0$. 

Comparing the effect of heteroscedasticity in one- versus two-sample setting, we see that 
\[
\rho_{\twosample}(\beta,1) = 2\rho(\beta,1),
\]
an observation first made in \citep{DonohoKipnis2020}. Interestingly, as shown in Figure~\ref{fig:phase_diagram_twosample}, this relation between $\rho_{\twosample}(\beta,s)$ and $\rho(\beta,s)$ does not hold when $s\neq 1$. Specifically, detection in the two-sample homeostatic setting ($s=1$) asymptotically requires twice the effect size. On the other hand, compared to the one-sample case, more than twice the effect size is needed for overdispersed mixtures ($s>1$) and less for underdispersed ones ($s<1$).


\subsection{Other Generalizations of Rare and Weak Models
\label{sec:discussion:conditions}
}
\cite{cai2014optimal} considered general rare and weak signal detection models characterized by the asymptotic behavior of the likelihood ratio between the mixture components on the moderate deviation scale which is similar to \eqref{eq:LR_APLC}. For rare and weak departures from the exponential distribution, the information-theoretic lower bound of Theorem~\ref{thm:powerlessness} generalizes \cite[Thm. 3]{cai2014optimal} by allowing for non-identically distributed coordinates and by providing conditions that involve integrated versions of the likelihood ratio. 

Another generalization of rare and weak signal detection models is provided by
\cite{arias2017distribution}, which considered a symmetric null distribution and proposed non-parametric HC- and Bonferroni-type tests that possess interesting optimality properties. Our RMD formulation applies to the setting of \cite{arias2017distribution} when the non-symmetric behavior of an individual test statistic under the alternative hypothesis is on the moderate deviation scale.

For the HC test, \cite{donoho2021impossibility} considered rare mixtures of P-values with non-null component $Q_i^{(n)}$ obeying
\begin{align*}
    \max_i \Pr_{P_i \sim Q_i^{(n)}} \left[P_i > 2q\log(n) \right] & = \max_i \exsub{X \sim \lognull} { L_i(X) \one_{\{X > 2q\log(n) \}} } \\
    & = n^{-\alpha'(q;\rho)+o(1)},
\end{align*}
for some continuous, non-negative bivariate function $\alpha'(q;\rho)$ that is increasing in $q$ and decreasing in $\rho$. They showed that HC of such P-values is powerless in the region 
\begin{align*}
    \Xi_\HC \equiv \left\{ (\rho,\beta)\,:\, \max_{q\in[0,1]} \left( \frac{1+q}{2} - \alpha'(q;\rho)\right)< \beta \right\}.
\end{align*}
The region $\Xi_\HC$ coincide with $\{ (\rho,\beta)\,:\, \rho< \rho(\beta,\sigma) \}$ in the RMD setting for which we have $\alpha'(q;\rho) = \max\{\alpha(q;\rho,\sigma),0\}$. 

Examples of rare and weak multiple testing settings with non-moderate departures include the sparse positive dependence model of \cite{arias2020dependence}, rare mixtures involving distributions of polynomial tails with a location shift of order $\mu_n(r)$ as the alternative studied in \cite{arias2019detection}, small Poisson means studied in \cite{arias2015sparse}, two-sample Poisson means in the low-counts case of \cite{DonohoKipnis2020}, and others \citep{jin2016rare}.

\ifsupp
\section*{Supplementary Materials}
\section{Technical Lemmas \label{sec:proofs_of_lemmas}}
\setcounter{equation}{0}
\begin{lemma}\label{lem:E_prob_to_Exp}
Let $\{P_i^{(n)}\}_{i=1}^n$ be a sequence of probability distributions, each $P_i^{(n)}$ has density whose support is contained in $[0,\infty)$. Fix $q>0$. If,
\begin{align}
    \lim_{n \to \infty} \max_{i=1,\ldots,n}  \frac{\left|\log\left(\frac{d P_i^{(n)}}{d\Exp(2)} (2q\log(n)) \right) \right| }{\log(n)}  = 0,
    \label{eq:exp_asym_lemma}
\end{align}
then
\begin{align}
    \lim_{n\to \infty} \max_{i=1,\ldots,n} \left| \frac{-\log  \Pr \left[P_i^{(n)} \geq 2q\log(n) \right]}{\log(n)}  - q \right|.
    \label{eq:E_prob_to_Exp}
\end{align}
\end{lemma}
\begin{proof} 
The assumption on the density of $P_i^{(n)}$ ensures that it is absolutely continuous with respect to $\lognull$. Fix $q>0$. We can write \eqref{eq:exp_asym_lemma} as 
\[
\frac{d P_i^{(n)}}{d\lognull} (2q\log(n)) = n^{o(1)},
\]
where $o(1) \to 0$ uniformly in $i$ for every fixed $q$. From
\[
\frac{d \lognull}{dx}(x) = \frac{e^{-x/2}}{2},
\]
we get
\begin{align*}
    \Prp{P_i^{(n)} \geq 2 q \log(n)} & = \int_{2 q \log(n)}^{\infty} \frac{d P_i^{(n)}}{dx} dx  \\
    & = \int_{2 q \log(n)}^{\infty} \frac{d P_i^{(n)}}{d \lognull} \frac{d\lognull}{dx} dx \\
    & = \int_{2 q \log(n)}^{\infty} n^{o(1)} e^{ -x/2} /2 dx \\
    & = n^{o(1)} e^{-q \log(n)}/2 = n^{-q + o(1)}
\end{align*}
This implies \eqref{eq:E_prob_to_Exp}.

\end{proof}

We will use the following lemma from \citep{cai2014optimal}, providing a particular version of Laplace's principle. 
\begin{lemma}{\cite[Lemma 3]{cai2014optimal}}
\label{lem:M}
Let $\left(X,\mathcal{F},\nu \right)$ be a measure space. Let $F\,:\,X \times \reals_+ \to \reals_+$ be measurable. Assume that 
\begin{align}
    \lim_{M\to\infty} \frac{\log F(x,M)}{M} = f(x)
    \label{eq:integrability_condition}
\end{align}
holds uniformly in $x\in X$ for some measureable $f\,:\,X\to \reals$. If
\[
\int_X \exp(M_0 f(x))d\nu(x)<\infty
\]
for some $M_0>0$, then
\begin{align}
    \lim_{M\to \infty} \frac{1}{M} \log \int_X F(x,M) d\nu(x) = \mathrm{ess}\sup_{x\in X} f(x). \nonumber
\end{align}
\end{lemma}

\begin{lemma}
\label{lem:LR_to_E}
Suppose that $\{Q_i^{(n)}\}_{i=1}^n$ satisfy \eqref{eq:APLCC_cond}, $\{E_i^{(n)}\}_{i=1}^n$ satisfy \eqref{eq:exp_asym}, $Q_i^{(n)}$ is absolutely continuous with respect to $E_i^{(n)}$, and $E_i^{(n)}$ is absolutely continuous with respect to the Lebesgue measure on $[0,\infty)$. Set 
\begin{align}
    L_i^{(n)}(x) := \frac{d Q_i^{(n)}}{d E_i^{(n)}}(x).  \nonumber
\end{align}
and 
\begin{align*}
    \alpha^*(q; r, \sigma) := \max_{y\in[r,q]} \left\{- 2\alpha(y;r,\sigma)+y \right\}.
\end{align*}
Assume that
\begin{align}
    \lim_{n\to \infty} \max_{i=1\ldots,n} \frac{ \left|\log\left( \frac{dQ_i^{(n)}}{d \chi^2(r,\sigma)} (2q\log(n))\right)\right|}{\log(n)} = 0,\quad \forall q\in (r,r+a),
    \label{eq:lem:LR_equivalence}
\end{align}
for some $a>0$ and $r>0$. Then, for any fixed $q\in (r,r+a)$,
\begin{align}
    \lim_{n\to \infty} \max_{i=1\ldots,n} \left|\frac{-\log\left(\exsub{X\sim Q_i^{(n)}}{L_i^{(n)}(X) \One{\{X \leq 2q\log(n) \}}} \right)}{\log(n)} -  \alpha^*(q; r ,\sigma) \right| = 0.
\label{eq:general_cond_B_proof}
\end{align}
\end{lemma}

\begin{proof} 
Fix $q \in (r, r+a)$. We have
\begin{align}
    & \exsub{X\sim Q_i^{(n)}}{ L_i^{(n)}(X) \One{ \{X \leq 2q \log(n)\} } }  = 
    \exsub{X\sim E_i^{(n)}}{ (L_i^{(n)})^2(X) \One{\{X \leq 2q \log(n) \} } } \nonumber \\
    & \qquad  = \int_0^{2q \log(n)} \left(\frac{d Q_i^{(n)}}{d E_i^{(n)}}(x)\right)^2  E_i^{(n)}(dx) \nonumber  \\
    & = 2\log(n) \int_0^q  \left(\frac{d Q_i^{(n)}}{d \E_i^{(n)}} (2 \log(n) y)\right)^2 E_i^{(n)}(2\log(n)dy)
    \label{eq:proof_L2_equiv1}
    \\
    & = \log(n)\int_0^q  \left(\frac{d Q_i^{(n)}}{d \E_i^{(n)}} (2 \log(n) y)\right)^2 e^{-y\log(n)(1+o(1))} dy
    \label{eq:proof_L2_equiv2}
    \\
    & = \log(n) \int_0^q n^{-2\alpha(y;r,\sigma)+2y+o(1)} \cdot n^{o(1)} \cdot
    n^{-y} dy = \int_0^q n^{-2\alpha(y;r,\sigma)+y+o(1)} dy,
    \label{eq:proof_L2_equiv3}
\end{align}
where \eqref{eq:proof_L2_equiv1} follows from the change of variables $x = 2y \log(n)$,  \eqref{eq:proof_L2_equiv2} follows from Lemma~\ref{lem:E_prob_to_Exp}, and \eqref{eq:proof_L2_equiv3} follows from \eqref{eq:lem:LR_equivalence}. Furthermore, $o(1)$ in \eqref{eq:proof_L2_equiv2}-\eqref{eq:proof_L2_equiv3} represents a sequence tending to zero uniformly in $i$ and $y\in[0,q]$. We now apply Lemma~\ref{lem:M} to \eqref{eq:proof_L2_equiv3} with $X=[r,q]$, $M=\log(n)$, $F(x,M)=n^{-2\alpha(x;r,\sigma)+x+o(1)}$, $f(x) = -2\alpha(x;r,\sigma)+x$, and $\nu$ the Lebesgue measure. We obtain:
\begin{align*}
    \lim_{n \to \infty} \sup_{i=1,\ldots,n} \frac{ \log\left( \exsub{X\sim Q_i^{(n)}}{ L_i^{(n)}(X) \One{\{X > 2q \log(n)\} } }
     \right)}{\log(n)} & = \max_{y\in[r,q]} \left\{- 2\alpha(y;r,\sigma)+y \right\} \\
     & = -\alpha^*(q; r ,\sigma)
\end{align*}
Equation \eqref{eq:general_cond_B_proof} follows. 
\end{proof} 

The following lemma summarizes the truncated likelihood ratio method from \citep{ingster2012nonparametric}.  
\begin{lemma}\label{lem:Ingster}
Consider testing  
\begin{align}
    H_0^{(n)}\,&:\,(X_1,\ldots,X_n)\sim P_0^{(n)} \nonumber
\end{align}
versus
\begin{align}
    H_1^{(n)}\,&:\,(X_1,\ldots,X_n)\sim P_1^{(n)} \nonumber
\end{align}
for $P_1^{(n)}$ that is absolutely continuous with respect to $P_0^{(n)}$. Denote by $L_n = \frac{dP_1^{(n)}}{dP_0^{(n)}}$ the likelihood ratio between $P_1^{(n)}$ and $P_0^{(n)}$. Suppose that there exists a sequence of sets $A^{(n)} \subset \reals^n$ such that 
\begin{align}
\label{eq:first_moment}
    1 - \exsub{H_0^{(n)}}{L_n(X_1,\ldots,X_n) \One{(X_1,\ldots,X_n) \in A^{(n)}} } \leq  o(1) 
\end{align}
while 
\begin{align}
\label{eq:second_moment}
    \exsub{H_0^{(n)}}{L_n^2(X_1,\ldots,X_n) \One{(X_1,\ldots,X_n) \in A^{(n)} } }  \leq 1 + o(1).
\end{align}
For any sequence of tests $\psi^{(n)}: \reals^{n} \to \{0,1\}$,  \begin{align*}
    \liminf_{n\to \infty} \left\{ \exsub{H_0^{(n)}}{\psi^{(n)}(X_1,\ldots,X_n)} + \exsub{H_1^{(n)}}{1-\psi^{(n)}(X_1,\ldots,X_n)} \right\} \geq 1.
\end{align*}
\end{lemma}
\begin{proof}
Set
\begin{align*}
    \tilde{L}_n := \tilde{L}_n(X_1,\ldots,X_n) := L_n (X_1,\ldots,X_n) \One{A^{(n)}}(X_1,\ldots,X_n).
\end{align*}    
Conditions \eqref{eq:first_moment} and \eqref{eq:second_moment} imply
\begin{align*}
    \exsub{H_0^{(n)}}{\tilde{L}_n} = \left(\exsub{H_0^{(n)}}{\tilde{L}_n^2}-1\right) - 2\left( \exsub{H_0^{(n)}}{\tilde{L}_n} - 1 \right) \leq o(1),
\end{align*}
hence $\tilde{L}_n(X) \to 1$ in probability under $H_0^{(n)}$. Next, for some $\psi^{(n)} : \reals^n \to \{0,1\}$ and $\epsilon>0$,
\begin{align*}
     & \exsub{H_0^{(n)}}{\psi^{(n)}} + \exsub{H_1^{(n)}}{1-\psi^{(n)}} \\
     & = \exsub{H_0^{(n)}}{\psi^{(n)} + L_n(1-\psi^{(n)})} \\
    & \geq \exsub{H_0^{(n)}}{\psi^{(n)} + \tilde{L}_n(1-\psi^{(n)})} 
    \\
    & \geq \exsub{H_0^{(n)}}{\psi^{(n)} + \tilde{L}_n(1-\psi^{(n)}) \mid |\tilde{L}_n-1| < \epsilon }\Prp{|\tilde{L}_n-1| < \epsilon} \\
    & \geq \exsub{H_0^{(n)}}{\psi^{(n)} + (1-\epsilon)(1-\psi^{(n)})  }\Prp{|\tilde{L}_n-1| < \epsilon} \\
    & \geq (1 -\epsilon)\Prp{|\tilde{L}_n-1| < \epsilon} = (1-\epsilon)(1+o(1)).
\end{align*}
As $\epsilon>0$ is arbitrary, we have that 
\begin{align*}
    \liminf_{\psi^{(n)}} \left\{ \exsub{H_0^{(n)}}{\psi^{(n)}} + \exsub{H_1^{(n)}}{1-\psi^{(n)}} \right\} \geq 1.
\end{align*} 
\end{proof}

\begin{lemma} \label{lem:under_null0}
Let $q\in(0,1]$ be fixed. Let $U_1,\ldots,U_n$ be $n$ independent RVs satisfying $\Prp{U_i \leq n^{-q}}=n^{-q}(1+a_{n,i}(q) )$, and denote by
\[
F_n(t) := \frac{1}{n}\sum_{i=1}^n \One{\{U_i \leq t\}}
\]
their empirical CDF. If $\bar{a}_n(q) := n^{-1} \sum_{i=1}^n a_{n,i}(q) \leq n^{\frac{q-1}{2}}$, then
\begin{align}
\label{eq:under_null0}
\Prp{ \sqrt{n}\frac{F_{n}(n^{-q})-n^{-q}}{\sqrt{n^{-q}(1-n^{-q})}} \geq \log(n) } \to 0
\end{align}
\end{lemma}

\begin{proof}
Denote $t_n = n^{-q}$. We have that $\ex{F_n(t_n)}=t_n(1+\bar{a}_n(q))$. If $\bar{a}_n(q) \leq 0$ for all $n \geq n_0$ for some $n_0$, then \eqref{eq:under_null0} holds. Otherwise, we assume without loss of generality that $r_n := \ex{F_n(t_n) - t_n} = n^{-q}\bar{a}_n(q) >0$ for all $n$, since the complementary case can be handled by considering only a sub-sequence with that property. Write
\begin{align*}
\Prp{ \sqrt{n}\frac{F_{n}(t_n)-t_n}{\sqrt{t_n(1-t_n)}} \geq \log(n)} & = \Prp{F_{n}(t_n)-t_n \geq (1+\delta) r_n},
\end{align*}
where 
\begin{align*}
    \delta := -1 + \frac{\sqrt{t_n(1-t_n)} \log(n) }{r_n\sqrt{n}} & \geq -1 + \log(n) n^{\frac{q-1}{2}} (\bar{a}_n(q))^{-1} \sqrt{1-1/n} \\
    & \geq -1 + \log(n) (1+o(1)).
\end{align*}
We have that $\delta \to \infty$. For $X$ the sum of $n$ independent Bernoulli RVs with $\mu = \ex{X}$, the Chernoff inequality \cite[Ch 4.]{mitzenmacher2017probability} says
\[
\Pr \left( X \geq (1+\delta)\mu \right) \leq \left( \frac{e^{-\delta}}{(1+\delta)^{1+\delta}} \right)^\mu \leq e^{-\mu \frac{\delta^2}{2+\delta}},\quad \mu = r_n,\quad  \delta\in (0,\infty). 
\]
We use this inequality with $X = n F_{n}(t) = \sum_{i=1}^n \One{\{U_i \leq t\}}$. For $n$ large enough such that $\delta>2$, we obtain
\begin{align*}
    & 
    -\log \Prp{ \sqrt{n}\frac{F_{n}(t_n)-t_n}{\sqrt{t_n(1-t_n)}} \geq \log(n) } \geq \frac{\delta^2 n}{2+\delta} r_n \geq \frac{\delta \cdot n}{2} r_n \\
    & \qquad \geq \frac{n}{2} \left(n^{-1/2}\log(n)-  r_n \right) = \frac{n^{0.5}}{2} \left(\log(n)- n^{-q+1/2} \bar{a}_n(q) \right) \to \infty.
\end{align*}
\end{proof}

\begin{lemma}{\cite[Lem. 5.7 ]{DonohoKipnis2020}}
\label{lem:Donoho_kipnis}
Let $\alpha(\cdot)$ and $\gamma(\cdot)$ be two real-valued functions $\alpha, \gamma : [0,\infty) \to [0,\infty)$. Let $q\in(0,1)$ and $\beta>0$ be fixed. Let $F_n(t)$ be the normalized sum of $n$ independent RVs. Suppose that
\begin{align*}
\ex{F_n(n^{-q})} = n^{-q+o(1)}(1-n^{-\beta}) +
n^{-\beta} n^{-\alpha(q)+o(1)}.
\end{align*}
Let $\{a_n\}_{n=1}^\infty$ be a positive sequence obeying $a_n n^{-\eta} \to 0$ for any $\eta>0$. If
\begin{align*}
    \delta(q)+\beta < \gamma(q),
\end{align*}
then
\[
\Prp{ n^{\gamma(q)} (F_{n}(n^{-q}) -n^{-q}) \leq a_n} \to 0, \qquad n\to \infty.
\]
\end{lemma}

\begin{lemma}\label{lem:F1}
Assume that $r < \rho_{\Bonf}(\beta,\sigma)$. Consider $p_1,\ldots,p_n$ as in \eqref{eq:hyp_log_n_appx}. For an interval $I\subset [0,1]$, define
\begin{align}
    T_I := \min_{i\,:\,p_{(i)} \in I} \frac{p_{(i)}}{i/n}.  \nonumber
\end{align}
For any $0<a<1$ and $q<1$, 
\begin{align}
    \Pr_{H_1^{(n)}} \left[ T_{(n^{-q},1]} \leq a \right] \to 0. \nonumber
\end{align}
\end{lemma}
\begin{proof}
Let $F_n(t) := n^{-1}\sum_{i=1}^n \One{p_i \leq t}$ be the empirical CDF of $p_1,\ldots,p_n$. Note that $i/n = F_n(p_{(i)})$, hence 
\begin{align}
 \frac{p_{(i)}}{i/n} \leq a \Longleftrightarrow F_n ( p_{(i)}) \geq p_{(i)}/a. \nonumber
\end{align}
Consequently, 
\begin{align}
    \Pr_{H_1^{(n)}} \left[ T_{(n^{-q},1]} \leq a \right] & \leq \sup_{t > n^{-q}} \Pr_{H_1^{(n)}} \left[ F_n(t) 
     \geq t/a \right] \nonumber \\
     & = \sup_{t > n^{-q}} \Pr_{H_1^{(n)}} \left[ nF_n(t) 
     \geq nt/a \right] \nonumber \\
     & = \sup_{t > n^{-q}} \Pr_{H_1^{(n)}} \left[ nF_n(t) 
     \geq \mathbb E_{{H_1^{(n)}}} \left[ {nF_n(t)}\right](1+\kappa) \right], \label{eq:lem:F1:proof:eta}
\end{align}
where 
\begin{align}
    \kappa := \kappa(n,a,t) := \frac{t}{a\mathbb E\left[F_n(t)\right]} - 1.  \nonumber
\end{align}
Let $U_i\sim \Unif(0,1)$ and $-2\log(X_i) \sim Q_i^{(n)}$, for $i=1,\ldots,n$. Using the parameterization $t_n = n^{-q'}$, $q' \leq q<1$, 
\begin{align}
    \mathbb E_{H_1^{(n)}}\left[ F_n(t_n)\right] & = \frac{1}{n}\sum_{i=1}^n \Pr_{H_1^{(n)}}\left[ p_i \leq n^{-q'} \right] \\
    & = (1-\epsilon_n) \Prp{ U_i \leq n^{-q'} }  + \epsilon_n \Prp{ X_i \leq n^{-q'} } \\
    & = (1-\epsilon_n)n^{-q'} + n^{-\beta} \cdot n^{-\alpha(q';r,\sigma) + o(1)} \\
    & = 1-\epsilon_n +  n^{-\alpha(q';r,\sigma)+o(1)-\beta}, \label{eq:last_disp1}
\end{align}
where the last transition follows from \eqref{eq:APLCC_cond}. The condition $r < \rho_{\Bonf}(\beta,\sigma)$ implies
\[
\sup_{q' \leq q<1} \left(q' - \beta - \alpha(q';r,\sigma) \right) \leq q - \beta - \alpha(1;r,\sigma) < 0.
\]
Therefore, from \eqref{eq:last_disp1} we get $\exsub{H_1^{(n)}}{F_n(t_n)}/n^{-q'}\to 1$. Because $F_n(t)$ has at most $n$ jumps, a simple grid argument implies 
\begin{align}
\inf_{t > n^{-q}} \frac{t}{\exsub{H_1^{(n)}}{F_n(t)}} = 1 + o(1). \nonumber
\end{align}
As $a < 1$, there exists $n_0(q)$ such that 
\[
\inf_{t>n^{-q}} \kappa(n,a,t) \geq (1/a - 1)/2 > 0,\qquad n \geq n_0(q).
\]
Set $b:= (1/a - 1)/2>0$. Applying the Chernoff inequality (c.f. \cite[Ch. 4]{mitzenmacher2017probability}) to the sum of independent RVs $n F_n(t)$ in  \eqref{eq:lem:F1:proof:eta}, we obtain
\begin{align*}
    \Pr_{H_1^{(n)}} \left[ T_{(n^{-q},1]} \leq a \right] &  \leq \sup_{t > n^{-q}} \exp\left\{-\frac{n}{a} \frac{\kappa^2}{1+\kappa} \mathbb E_{H_1^{(n)}}\left[F_n(t)\right]  \right\} \\
    & \leq \exp \left\{- \frac{n}{2a}  \inf_{t > n^{-q}} \frac{\kappa^2}{1+\kappa}   
    E_{H_1^{(n)}}\left[F_n(t)\right]  \right\} \\ 
    & = \exp \left\{- \frac{1}{2a} \frac{b^2}{1+b} n^{1-\alpha(q;r,\sigma) +o(1)-\beta} \right\} \to 0,
\end{align*}
where the last transition follows because $r < \rho_\Bonf(\beta,\sigma)$ implies $\beta+\alpha(q;r,\sigma) \leq \beta+\alpha(1;r,\sigma) < 1$. 
\end{proof}

\begin{lemma}
    \label{lem:Poisson}
    Let $\{a_n\}$,$\{b_n\}$, and $\{\lambda_n\}$ be non-negative sequences such that, as $n \to \infty$, 
    $a_n \to \infty$, $\lambda_n \to \infty$, $a_n/\lambda_n \to 0$, and $a_n/b_n \to c$ for some $c>1$. For $\lambda' = \lambda_n + \sqrt{\lambda_n b_n}$ and $\Upsilon_{\lambda'} \sim \Pois(\lambda')$, 
    \begin{align}
        \lim_{n \to \infty} \frac{\Prp{-2\log \Pp(\Upsilon_{\lambda'};\lambda_n) \geq a_n}} {(\sqrt{a_n} - \sqrt{b_n})^2} = -\frac{1}{2}. \nonumber
    \end{align}
\end{lemma}

\begin{proof}
We first develop a moderate deviation estimate for the Poisson survival function. From
\begin{align*}
    \Pp(x;\lambda) & = e^{-\lambda} \sum_{k=x}^\infty \frac{\lambda^k}{k!} \\
    & = e^{-\lambda} \frac{\lambda^x}{x!} \left( 1 + \frac{\lambda^{x+1}}{(x+1)} + \frac{\lambda^{x+2}}{(x+1)(x+2) + \ldots} \right),
\end{align*}
we get
    \begin{align*}
        -\log \Pp(x;\lambda) = \lambda - x \log(\lambda) + \log \Gamma(x) + R(x;\lambda),
    \end{align*}
where $\Gamma(x)$ is the Gamma function and
\[
R(\lambda;x) := \log\left(1 + \frac{\lambda^{x+1}}{(x+1)} + \frac{\lambda^{x+2}}{(x+1)(x+2) + \ldots} \right) \leq -\log \left( 1- \frac{\lambda}{1+x} \right) = O(\lambda/x). 
\]
Furthermore, 
\[
\log \Gamma(x) = (x-\frac{1}{2})\log(x) - x + O(1/x) 
\]
Therefore, for $x > \lambda > 1$, we have that $t \leq -\log \Pp(x;\lambda)$ iff
\[
t \leq -(x-\lambda) + x \log(x/\lambda) + O(\lambda/x) = -(x-\lambda) x + x \left( \frac{x-\lambda}{\lambda} \right)  + O(\lambda/x),
\]
if and only if $
0 = x^2 - 2x \lambda + \lambda^2 - t \lambda + o(\lambda/x)$. Solving the last equation for $x>0$, we obtain $t \leq -\log \Pp(x;\lambda)$ if
\begin{align}
x & \geq \lambda + \sqrt{ t \lambda} + O(\sqrt{\lambda/x}).
\label{eq:lem:Poisson:2}
\end{align}

Next, consider the event $A = \{\Upsilon_{\lambda'} \geq \lambda_n\}$. We have
\begin{align}
\Prp{-\log \Pp(\Upsilon_{\lambda'};\lambda_n) \geq a_n | A} & \overset{a}{=} \Prp{ \Upsilon_{\lambda'} \geq \lambda_n + \sqrt{\lambda_n a_n} + O(\sqrt{\lambda_n/\Upsilon_{\lambda'})} |A} \nonumber \\
= & \Prp{\Upsilon_{\lambda'} \geq \lambda_n + \sqrt{a_n \lambda_n} + O(1) |A} \nonumber \\
= & \Prp{\Upsilon_{\lambda'} \geq (\lambda_n + \sqrt{a_n \lambda_n})(1+o(1)) |A} \nonumber \\
\overset{b}{=} & \Prp{\Upsilon_{\lambda'} \geq (\lambda' + \sqrt{\lambda'}\left(\sqrt{a_n} - \sqrt{b_n}\right)(1+o(1)) |A} \nonumber \\
& = \Prp{\Upsilon_{\lambda'} \geq \lambda' + \sqrt{\lambda'}\sqrt{c_n} |A}
\label{eq:lem:Poisson:1},
\end{align}
where $\{c_n\}$ is a sequence satisfying
\begin{align}
\frac{\sqrt{c_n}}{{\sqrt{a_n}} - \sqrt{b_n}} \to 1\quad \text{as}\quad n \to \infty. 
\label{eq:lem:Poisson:3}
\end{align}
In the arguments leading to \eqref{eq:lem:Poisson:1}, $(a)$ is due to \eqref{eq:lem:Poisson:2} and $(b)$ is due to 
\begin{align*}
\lambda_n + \sqrt{\lambda_n a_n}& = \lambda' + \sqrt{\lambda'}\left(\sqrt{a_n} - \sqrt{b_n} \right)\sqrt{ \lambda_n/\lambda'} \nonumber \\
& = \lambda' + \sqrt{\lambda'}\left(\sqrt{a_n} - \sqrt{b_n} \right)(1 + o(1)),
\end{align*}
the last transition because $b_n/\lambda_n \to 0$. 

Since $\sqrt{\lambda_n b_n/\lambda'} \to \infty$, the normal approximation $\Upsilon_{\lambda'} \sim \Ncal(\lambda',\lambda')$ implies
\[
\Prp{A} \sim \Prp{ \sqrt{\lambda'}Z + \lambda' \geq \lambda_n} = \Prp{ Z \geq -\sqrt{\lambda_n b_n/\lambda'} } \to 1.
\]
We obtain 
\begin{align*}
\lim_{n \to \infty} \frac{\log\Prp{-\log \Pp(\Upsilon_{\lambda'};\lambda_n) \geq a_n}}{(\sqrt{a_n} - \sqrt{b_n})^2 } & = 
\lim_{\lambda_n \to \infty} \frac{\log\Prp{-\log \Pp(\Upsilon_{\lambda'};\lambda_n) \geq a_n|A}}{(\sqrt{a_n} - \sqrt{b_n})^2 } \\
& = \lim_{\lambda_n \to \infty} \frac{\log\Prp{\Upsilon_{\lambda'} \geq \lambda_n + \sqrt{a_n \lambda_n} |A}}{(\sqrt{a_n} - \sqrt{b_n})^2 } \\
& = \lim_{n \to \infty} \frac{\log\Prp{\Upsilon_{\lambda'} \geq \lambda_n + \sqrt{a_n \lambda_n}}}{(\sqrt{a_n} - \sqrt{b_n})^2 } \\
& \overset{c}{=} \lim_{n \to \infty} \frac{\log\Prp{\Upsilon_{\lambda'} \geq \lambda'
+ \sqrt{\lambda' c_n }}}{(\sqrt{a_n} - \sqrt{b_n})^2 } \\
& \overset{d}{=} \lim_{n \to \infty} \frac{\log\Prp{\Upsilon_{\lambda'} \geq \lambda'
+ \sqrt{\lambda' c_n }}}{c_n} \\
& \overset{e}{=} -\frac{1}{2}
\end{align*}
where $(c)$ is due to \eqref{eq:lem:Poisson:1}, $(d)$ follows from \eqref{eq:lem:Poisson:3}, and in $(e)$ we used the following moderate deviation estimate for a Poisson RV from \citep{arias2015sparse}.
\begin{lemma}{\cite[Lemma 2]{arias2015sparse}}
\label{lem:arias2015}
Let $c : (0,\infty) \to (0, \infty)$ be such that $c(\lambda)\to \infty$ and $c(\lambda)/\lambda \to 0$ as $\lambda \to \infty$. Then
\[
\lim_{\lambda \to \infty} \frac{\log \left( \Upsilon_{\lambda} \geq \lambda + \sqrt{\lambda c(\lambda)} \right)}{ c(\lambda)} = \frac{-1}{2}
\]
\end{lemma}
This completes the proof of Lemma~\ref{lem:Poisson}.
\end{proof}

\section{Proofs of Results in Section~\ref{sec:APLC}} 
\setcounter{equation}{0}


\subsection{Proof of Theorem~\ref{thm:powerlessness}
\label{sec:proof:powerlessness}
}

Note that for $a<r$, 
\[
\frac{a}{2} - \alpha(a; r, \sigma) \leq \frac{r}{2} - \alpha(r; r, \sigma) = \frac{r}{2},
\]
hence the maximum in \eqref{eq:rho_as_minmax} is not attained in the interval $q \in [0,r)$. Hence, by \eqref{eq:rho_as_minmax}, for $r<\rho^*(\beta,\sigma)$ there exists $\delta>0$ such that 
\begin{align}
\label{eq:HC_cond0}
      \max_{q\in [r,1]} \left( \frac{1+q}{2} - \alpha^*(q;r,\sigma)\right) + \delta - \beta < 0.
\end{align}
In particular, 
\begin{align}
\label{eq:HC_cond1}
      1 - \beta - \alpha^*(1;r,\sigma) \leq -\delta,
\end{align}
and, by continuity of $q\to \alpha^*(q;r,\sigma)$, there exists $\eta \in (0,\gamma)$ such that 
\begin{align}
\label{eq:HC_cond2}
    \quad 1 - 2\beta - \alpha^*(1+\eta;r,\sigma) < - \delta. 
\end{align}

Fix $\delta>0$ satisfying \eqref{eq:HC_cond0}, and let $\eta>0$ satisfy \eqref{eq:HC_cond2}. We now use Lemma~\ref{lem:Ingster} with 
\begin{align*}
P_0^{(n)} = \prod_{i=1}^n E_i^{(n)}, \qquad P_1^{(n)} = \prod_{i=1}^n\left[ (1-\epsilon) E_i^{(n)} + \epsilon Q_i^{(n)} \right],
\end{align*}
where $\{E_i^{(n)} \}$ satisfy \eqref{eq:exp_asym}, and 
\begin{align*}
    A^{(n)} = \prod_{i=1}^n \{X_i \leq 2(1+\eta)\log(n)\}.
\end{align*}
We have 
\begin{align}
    \tilde{L}_n = \prod_{i=1}^n \bar{L}_i^{(n)}(X_i)\One{\{X_i \leq 2(1+\eta)\log(n) \}},
    \label{eq:L_n_tilde}
\end{align}
where 
\begin{align}
\bar{L}_i^{(n)}(x):= (1-\epsilon_n)+\epsilon_n L_i^{(n)} = 1+\epsilon_n (L_i^{(n)}(x)-1),
\label{eq:ell_n}
\end{align}
and
\begin{align*}
    L_i^{(n)}(x) := \frac{dQ_i^{(n)}}{d E_i^{(n)}}(x) . 
\end{align*}

Henceforth, all expectations are with respect to $X_i\sim E_i^{(n)}$ unless otherwise specified. For the first moment, since $\ex{L_i^{(n)}(X_i)}=1$, we have
\begin{align}
\ex{\tilde{L}_n} & = \prod_{i=1}^n \left( 1 - a_{n,i} \right),
\label{eq:proof:method2:L1_prod}
\end{align}
where, 
\[
a_{n,i} :=  \ex{ \bar{L}_i^{(n)}(X_i) \One{\{ X_i > 2(1+\eta) \log(n)\}}}.
\]
Consider
\begin{align}
    a_{n,i} & 
    = \Pr_{X_i \sim E_i^{(n)}}\left[X_i \geq 2(1+\eta)\log(n) \right] 
     + \epsilon_n \ex{\left( L_i^{(n)}(X_i)  - 1\right) \One{\{ X_i > 2(1+\eta) \log(n)\}} } \nonumber \\
     & \leq \Pr_{X_i \sim E_i^{(n)}}\left[X_i \geq 2(1+\eta)\log(n) \right] 
     + \epsilon_n \ex{ L_i^{(n)}(X_i) \One{\{ X_i > 2(1+\eta) \log(n)\}}} \nonumber \\
     & 
    = n^{-(1+\eta)+o(1)} 
     + n^{-\beta} n^{-\alpha^*(1+\eta ;r,\sigma)+o(1)},
     \label{eq:proof:method2:L1}
\end{align}
where the last transition follows from \eqref{eq:exp_asym} and from  \eqref{eq:APLCC_cond}. It follows from \eqref{eq:HC_cond1} that $a_{n,i}=o(1/n)$, hence \eqref{eq:proof:method2:L1_prod} converges to $1$ and the first moment condition of Lemma~\ref{lem:Ingster} holds.

As for the second moment, we have
\begin{align}
\ex{\tilde{L}_n^2} & = \prod_{i=1}^n \ex{ \left( (1-\epsilon_n)^2 + 2\epsilon_n(1-\epsilon_n) L_i^{(n)}(X_i) + \epsilon_n^2  (L_i^{(n)}(X_i))^2 \right) \One{\{ X_i \leq 2(1+\eta)\log(n) \}} }\nonumber \\
& = \prod_{i=1}^n \left( (1-\epsilon_n)^2 + 2 \epsilon_n (1-\epsilon_n) \ex{L_i^{(n)}(X_i)} +  \epsilon_n^2 \ex{(L_i^{(n)}(X_i))^2  \One{\{ X_i \leq 2(1+\eta)\log(n) \} }} \right) \nonumber \\
& \leq \prod_{i=1}^n \left( 1-\epsilon_n^2 + \epsilon_n^2 \ex{(L_i^{(n)}(X_i))^2  \One{\{ X_i \leq 2(1+\eta)\log(n) \}} } \right) \leq \prod_{i=1}^n \left( 1 + b_{n,i} \right), \label{eq:thm:proof:L2}
\end{align}
where 
\[
b_{n,i} := \epsilon_n^2 +  \epsilon_n^2 \ex{(L_i^{(n)}(X_i))^2 \One{\{ X_i \leq 2(1+\eta)\log(n) \}} }.
\]
By \eqref{eq:general_cond_B},
\begin{align*}
    \ex{(L_i^{(n)}(X_i))^2  \One{\{ X_i \leq 2(1+\eta)\log(n) \}} } & = \exsub{X_i \sim Q_i^{(n)}}{ L_i^{(n)}(X_i) \One{\{ X_i \leq 2(1+\eta)\log(n) \}} }\\
    & = n^{ - \alpha^*(1+\eta\,;\, r, \sigma) + o(1)}.
\end{align*}
It follows from \eqref{eq:HC_cond2} that
\begin{align*}
n\cdot b_{n,i} & = n^{1-2\beta} + n^{1 - 2\beta  - \alpha^*(1+\eta\,;\,r, \sigma) \} + o(1)} < n^{1-2\beta} + n^{-\delta},
\end{align*}
which implies $b_{n,i} = o(1/n)$ since $\beta > 1/2$. We conclude that \eqref{eq:thm:proof:L2} converges to $1$ hence the second-moment condition of Lemma~\ref{lem:Ingster} holds and the proof of Theorem~\ref{thm:powerlessness} is completed. \qed

\subsection{Proof of Corollary~\ref{cor:powerlessness}
\label{sec:proof:cor:powerlessness}
}
Condition \eqref{eq:general_cond_A} follows from Lemma~\ref{lem:LR_to_E}. Condition \eqref{eq:general_cond_B} follows from \eqref{eq:APLCC_cond}.

\subsection{Proof of Theorem~\ref{thm:powerfulness}
\label{sec:proof:thm:HC}
}
Under \eqref{eq:exp_asym} and \eqref{eq:HC_cond}, Lemma~\ref{lem:under_null0} implies that 
\begin{align}
    \Pr_{H_0^{(n)}} \left[\HC_n^* > \log(n) \right] \to 0.  \nonumber
\end{align}
Therefore, it is enough to show that $
    \Pr_{H_1^{(n)}} \left[\HC_n^* \leq \log(n) \right] \to 0$. Set
\[
F_n(t) := \frac{1}{n} \sum_{i=1}^n \One{p_i \leq t}. 
\]
note that \eqref{eq:exp_asym} and \eqref{eq:APLCC_cond} imply
\begin{align*}
    \exsub{H_1^{(n)}}{F_n(n^{-q})} = n^{-q + o(1)}(1-n^{-\beta}) + n^{-\beta} n^{-\alpha^*(q;r,\sigma)+o(1)}.
\end{align*}
Additionally, 
\begin{align}
    \HC_n^* = \max_{1 \leq i \leq n \gamma_0} 
\sqrt{n} \frac{\frac{i}{n}-p_{(i)}}{\sqrt{p_{(i)} (1-p_{(i)})}}
    = \sup_{1/n \leq u \leq \gamma_0 n} \sqrt{n} \frac{F_n(u)-u}{\sqrt{u (1-u)}}, \nonumber
\end{align}
because the supremum occurs at jumping points of $F_n(u)$. Hence, provided $\gamma_0 < 1/2$, we surely have
\begin{align}
\HC_n^*
    \geq \sqrt{n} \frac{F_n(t)-t}{\sqrt{t (1-t)}} \geq \sqrt{\frac{n}{t}} \left( F_n(t) - t \right),\quad \forall t\in[1/n,1/2). \nonumber
\end{align}
Setting $t_n=n^{-q}$ for $q\leq 1$, we obtain:
\begin{align}
    \Pr_{H_1^{(n)}} \left[ \HC_n^* \leq \log(n) \right] & \leq \Pr_{H_1^{(n)}} \left( \sqrt{n} \frac{F_{n}(t_n)-t_n}{\sqrt{t_n (1-t_n)}} \leq \log(n) \right) \nonumber  \\
    & \leq \Pr_{H_1^{(n)}} \left[ n^{\frac{q+1}{2}}(F_n(t_n) -t_n) \leq \log(n) \right]. \label{eq:proof_main1}
\end{align}
Apply Lemma~\ref{lem:Donoho_kipnis} with $\delta(q) = \alpha^*(q;r,\sigma)$, $\gamma(q) = (q+1)/2$, and $a_n = \log(n)$ to conclude that \eqref{eq:proof_main1} goes to zero as $n\to \infty$. Theorem~\ref{thm:powerfulness} follows. \qed

\subsection{Proof of Theorem~\ref{thm:BJpowerfulness}}

The proof is similar to the proof of \cite[Thm. 4.4]{moscovich2016exact}. In particular, we use:
\begin{lemma}
\label{lem:Moscovich}
\cite[Cor. A1]{moscovich2016exact} 
Let $\{\alpha_n\}_{n=1}^\infty$ be a sequence converging to infinity. Let $\mu_n$, $\sigma^2_n$, and $f_n$ denote the mean, variance and density of  ~$\Beta(\alpha_n,n-\alpha_n+1)$, respectively. Let $g(n)$ be any positive function satisfying $g(n) = o(\min\{\alpha_n, n-\alpha_n\})$ as $n\to \infty$. Then,
\begin{align}
    f_n(\mu_n + \sigma_n \cdot t) \geq \frac{e^{-t^2/2}}{\sqrt{2\pi} \cdot \sigma_n } \left( 1 - \frac{t^3}{\sqrt{g(n)}} - \frac{1}{g(n)} \right). \nonumber
\end{align}
\end{lemma}
Recall that $M_n^- = \min_{i=1,\ldots,n} \pi_{i}$, where 
\begin{align*}
    \pi_i = \Pr\left[ \Beta(i,n-i+1) \leq p_{(i)} \right],\qquad i=1,\ldots,n.
\end{align*}
We use the sequence $t_n = 1/n$ to separate $H_0^{(n)}$ from $H_1^{(n)}$. The limiting distribution of $M_n$ under $H_0$ satisfies (see \cite{gontscharuk2015intermediates} and \cite[Thm 4.1]{moscovich2016exact}),
\begin{align*}
    \Pr_{H_0} \left[ M_n^- \leq \frac{x}{2\log(n) \log \log(n)} \right] \to 1-e^{-x},
\end{align*}
from which it follows that 
\begin{align}
    \Pr_{H_0} \left[ M_n^- \leq t_n \right] \to 0. \label{eq:proof:mosc:H0}
\end{align}

For $X \sim \Beta(i, n-i+1)$, set
\begin{align*}
 \mu_i := \ex{X} = \frac{i}{n+1},\qquad  \sigma_i^2 := \Var{X} = \frac{i(n-1+1)}{(n+1)^2(n+2)},
\end{align*}
hence, for $x \in \reals$,
\begin{align}
\label{eq:beta_x}
\frac{\mu_i - x}{\sigma_i} = 
    \sqrt{n}\frac{i/n - x}{\sqrt{\frac{i}{n} \left(1-\frac{i}{n}\right) }}(1+o(1)).
\end{align}

The proof of Theorem~\ref{thm:powerfulness} in Section~\ref{sec:proof:thm:HC} implies in particular
\begin{align}
\label{eq:last_display2}
\Pr_{H_1^{(n)}}\left[\max_{i=1,\ldots,n} \sqrt{n}\frac{i/n - p_{(i)}}{\sqrt{\frac{i}{n} \left(1-\frac{i}{n}\right) }} \geq \log(n) \right] \to 1. 
\end{align}
By \eqref{eq:beta_x} and \eqref{eq:last_display2}, for any $\delta>0$ there exists $n_0(\delta)$ and $i^* \in \{1,\ldots,n\}$ such that \begin{align}
\tau^* := \frac{\mu_{i^*} - p_{(i^*)}}{\sigma_{i^*}} \geq \sqrt{ 2\log(n)}, \nonumber
\end{align}
with probability at least $1-\delta$.  %
Denote by $f_i$ the density $f_i\,:[0,1] \to \reals^+$ of $\Beta(i,n-i+1)$. We have
\begin{align}
    \pi_{i^*} & = \int_{0}^{p_{(i^*)}} f_{i^*}(x)dx \nonumber \\
    & = \sigma_{i^*} \int_{-\mu_{i^*}/\sigma_{i^*}}^{\tau^*} f_{i^*}(\mu_{i^*}+\sigma_{i^*}t)dt \nonumber \\
    & \leq \sigma_{i^*} \int_{-\infty}^{\tau^*} f_{i^*}(\mu_{i^*}+\sigma_{i^*}t)dt \nonumber \\
    & \leq \int_{\tau^*}^\infty \frac{1+o(1)}{\sqrt{2\pi}}e^{-x^2/2} dx \label{eq:proof:mosco:1}\\
    & = (1-\Phi(\tau^*))(1+o(1)) \sim \frac{1}{\tau^*}e^{-\tau^{*2}/2} \label{eq:proof:mosco:2},
\end{align}
where \eqref{eq:proof:mosco:1} follows from Lemma~\ref{lem:Moscovich} and \eqref{eq:proof:mosco:2} is due to Mills' ratio. Consequently, 
\begin{align}
    \pi_{i^*} \leq n^{-1}, \quad \text{for}\quad n\geq n_0(\delta), \nonumber
\end{align}
with probability at least $1-\delta$. Hence, for $n \geq n_0(\delta)$,
\begin{align*}
\Pr_{H_1^{(n)}} \left[M_n \leq t_n \right] \geq \Pr_{H_1^{(n)}} \left[M_n^{-} \leq n^{-1} \right] \geq
\Pr_{H_1^{(n)}} \left[\pi_{i^*} \leq n^{-1}  \right] \geq 1-\delta.
\end{align*}
Combining the last equation with \eqref{eq:proof:mosc:H0}, we conclude that the sequence of thresholds $t_n = 1/n$ separates $H_0^{(n)}$ from $H_1^{(n)}$ as advertised. \qed

\subsection{Proof of
Theorem~\ref{thm:Bonferroni} }
Set $\eta := 1 - \alpha(1;r,\sigma)-\beta$. 
The condition  $r>\rho_\Bonf(\beta,\sigma)$ implies $\eta>0$. By continuity of $q \to \alpha(q;r,\sigma)$, there exits $\delta>0$ such that 
\begin{align}
\label{eq:Bonf:proof:cond}
    1 - \alpha(1+\delta;r,\sigma) -\beta > \eta/2. 
\end{align}
For the statistic $p_{(1)} := \min_{i=1\ldots,n} p_i$, we show that, along the sequence of thresholds $a_n  = n^{-(1+\eta/2)}$, we have  $\Pr_{H_0^{(n)}}(p_{(1)} > a_n)\to 1$ while $\Pr_{H_1^{(n)}}(p_{(1)} > a_n)\to 0$. Indeed,
\begin{align}
    \Pr_{H_0^{(n)}} \left[ p_{(1)} \leq a_n \right] & = 1- \prod_{i=1}^n  \Pr_{H_0} \left[ p_i > a_n \right] \nonumber \\
    & = 1- \left(1 - a_n\cdot n^{o(1)} \right)^n \label{eq:Bonf:proof:H0} \\
    & = 1- \left(1 - n^{-(1+\eta/2)+o(1)} \right)^n
    \to 0, \nonumber
\end{align}
where \eqref{eq:Bonf:proof:H0} follows from \eqref{eq:exp_asym}. 
On the other hand, 
\begin{align}
    \Pr_{H_1^{(n)}}\left[ p_{(1)} \leq a_n \right] & = 1 - \prod_{i=1}^n  \Pr_{H_1^{(n)}}\left[ p_i > a_n \right]\nonumber \\
    & = 1 - \prod_{i=1}^n  \left(1 - \Pr_{H_1^{(n)}} \left[ p_i \leq a_n \right] \right),
    \label{eq:Bonf:proof:H1}
\end{align}
hence it is enough to show that $\Pr_{H_1^{(n)}} \left[ p_i \leq a_n \right] > n^{-1+\eta/2 + o(1)}$ uniformly in $i$. 
For $i=1,\ldots,n$, let $X_i$ be a RV with law $-2\log(X_i) \overset{D}{=} Q_i^{(n)}$. We have:
\begin{align}
\Pr_{H_1^{(n)}} \left[ p_i \leq a_n \right] & = (1-\epsilon_n) a_n\cdot n^{o(1)} + \epsilon_n \Pr \left[X_i \leq a_n\right]  \nonumber \\
& \ge \epsilon_n \Pr \left[X_i \leq  a_n\right] \nonumber \\
& = n^{-\beta- \alpha(1+\delta;r,\sigma) + o(1) }
\label{eq:Bonf:proof:H1_3}\\
& \ge n^{-1 + \eta/2 + o(1) }
\label{eq:Bonf:proof:H1_4},
\end{align}
where \eqref{eq:Bonf:proof:H1_3} follows from \eqref{eq:APLCC_cond}, and \eqref{eq:Bonf:proof:H1_4} follows from \eqref{eq:Bonf:proof:cond}. 
\qed

\subsection{Proof of Theorem~\ref{thm:FDR} }
The proof is similar to the proof of Theorem~1.4 in \citep{donoho2004higher}. The main idea is to establish the following claims:
\begin{enumerate}
    \item[(i)] Inference based on FDR thresholding ignores P-values in the range $(n^{-q},1]$, for $q<1$. 
    \item[(ii)] When $r < \rho_\Bonf(\beta,\sigma)$, P-values smaller than $n^{-q}$ under $H_1^{(n)}$ are as frequent as under $H_0^{(n)}$. 
\end{enumerate}
In order to establish (i) and (ii), define, for an interval $I\subset [0,1]$, 
\begin{align}
    T_I := \min_{i\,:\,p_{(i)} \in I} \frac{p_{(i)}}{i/n}.  \nonumber
\end{align}
For some $q>0$ and a sequence $\{a_n\}_{n=1}^\infty$ of threshold values
with $\liminf_{n\to\infty} a_n = 0$, 
\begin{align}
& \left|\Pr_{H_0^{(n)}}\left[\text{FDR rejects} \right] - \Pr_{H_1^{(n)}} \left[\text{FDR rejects} \right] \right| =  \left|\Pr_{H_0^{(n)}}\left[T_{[0,1]} < a_n \right] - \Pr_{H_1^{(n)}}\left[T_{[0,1]} < a_n \right]  \right|  \nonumber \\
    &\qquad \leq \Pr_{H_1^{(n)}} \left[T_{(n^{-q},1]} < a_n \right] + \Pr_{H_0^{(n)}} \left[T_{(n^{-q},1]} < a_n \right]
    \label{eq:FDR:proof:2} \\
    & \qquad \qquad  + \left|\Pr_{H_1^{(n)}} \left[T_{(0,n^{-q}]} < a_n \right] - \Pr_{H_0^{(n)}} \left[T_{(0,n^{-q}]} < a_n \right] \right| \label{eq:FDR:proof:1}. 
\end{align}
Note that the terms in \eqref{eq:FDR:proof:2} are associated with (i) while \eqref{eq:FDR:proof:1} is associated with (ii). \par
Lemma~\ref{lem:F1} implies that the terms in \eqref{eq:FDR:proof:2} vanish as $n\to \infty$. We now focus on the term \eqref{eq:FDR:proof:1}. Let $I \subset \{1,\ldots,n\}$ be a random set 
such that $i \in I$ with probability $\epsilon_n=n^{-\beta}$. Considering this randomness, an equivalent way of specifying $H_1^{(n)}$ is
\begin{align}
    -2\log(p_i) \sim \begin{cases} Q_i^{(n)} & i \in I \\
    \lognull & i\neq I,
    \end{cases}
    \qquad i=1,\ldots,n. \nonumber
\end{align}
For $i=1,\ldots,n$, let $X_i$ be a RV satisfying $-2 \log(X_i) \overset{D}{=} Q_i^{(n)}$. Choose $r<q<1$ such that 
\begin{align}
    1-\alpha(q;r,\sigma)-\beta + \delta < 0
    \label{eq:Bonf_cond}
\end{align}
for some $\delta>0$, which is possible since $r < \rho_\Bonf(\beta,\sigma)<1$. Consider the event:
\[
E_n^{q} := \{ p_i \leq n^{-q} \text{ for some } i \in I \}.
\]
Conditioned on the event $|I|=M$, \begin{align}
    \Pr\left[ E_n^{q} \mid |I|=M\right] & = \Pr \left[ \min_{i=1,\ldots,n} X_i  \leq n^{-q} \mid |I|=M \right] \nonumber \\
    & \leq 1-\left(1-n^{-\alpha(q;r,\sigma)+o(1)}\right)^M \label{eq:FDR:proof:3}  \\
    & \leq M \cdot n^{-\alpha(q;r,\sigma)+o(1)} \label{eq:FDR:proof:4} 
\end{align}
where \eqref{eq:FDR:proof:3} follows from \eqref{eq:APLCC_cond} and \eqref{eq:FDR:proof:4} follows from the inequality $M \cdot \log(1+x) > \log(1+ M x)$, $x\geq -1$. As $M\sim \Bin(n,\epsilon_n)$, we have $\Prp{M < n^{1+\delta/2} \epsilon_n} = \Prp{M < n^{1+\delta/2-\beta} }  \to 1$, e.g. by the Chernoff bound. Consequently, for any $b>0$,
\begin{align*}
     \Prp{ M \cdot n^{-\alpha(q;r,\sigma)+o(1)} > b} & \leq o(1) + \One{\{ n^{1 + \delta/2 -\beta -\alpha(q;r,\sigma)+o(1)} > b \}} \to 0
\end{align*}
where the last transition is due to \eqref{eq:Bonf_cond}. It follows that $\Pr\left[ E_n^{q} \right]\to 0$. From here, and using
\begin{align*}
    \Pr_{H_1^{(n)}} \left[ T_{[0, n^{-q})} < a_n \mid (E_n^{q})^c \right] = \Pr_{H_0^{(n)}} \left[ T_{[0, n^{-q})} < a_n \right],
\end{align*}
we get
\begin{align*}
\Pr_{H_1^{(n)}} \left[T_{(0,n^{-q}]} < a_n \right]  & = \Pr\left[ (E_n^{q})^c \right] \Pr_{H_1^{(n)}} \left[ T_{[0, n^{-q})} < a_n \mid (E_n^{q})^c \right] \\
& \qquad +  \Pr\left[ E_n^{q} \right] \Pr_{H_1^{(n)}} \left[ T_{[0, n^{-q})} < a_n \mid E_n^{q} \right] \\
& = \Pr_{H_1^{(n)}} \left[ T_{[0, n^{-q})} < a_n \mid (E_n^{q})^c \right](1+o(1)) + o(1) \\
& = \Pr_{H_0^{(n)}} \left[ T_{[0, n^{-q})} < a_n  \right] + o(1).
\end{align*}
All this implies that  \eqref{eq:FDR:proof:1} vanishes as $n\to \infty$. The proof is completed. \qed

\subsection{Proof of Theorem~\ref{thm:Fisher} }
We consider first the case in which $\{-2\log(p_i)\}_{i=1}^n$ follow \eqref{eq:hyp_log_n} and later extend our arguments to the general case \eqref{eq:hyp_log_n_appx}. 
Since $F_n \sim \chi_{2n}^2$, under the null in \eqref{eq:hyp_log_n} we have 
\begin{align}
    \ex{F_n | H_0^{(n)}} = 2n,\qquad \Var{F_n | H_0^{(n)}} = 4n.  \nonumber
\end{align}
As $F_n$ is asymptotically normal, it is enough to show that 
\begin{align}
\ex{F_n | H_1^{(n)}}\sim 2n(1+o(1/\sqrt{n})), \quad \text{and}\quad \Var{F_n | H_1^{(n)}}\sim 4n(1+o(1)).
    \label{eq:proof:Fisher:to_show}
\end{align}
For $X \sim \chi^2(r,\sigma)$, we have
\begin{align}
\ex{X} = \mu_n(r)^2 + \sigma^2,\qquad \ex{X^2 } = \mu_n^4(r) +  4\mu_n^2(r) \sigma^2 + 3\sigma^4, \nonumber
\end{align}
and $\Var{X} = 2\mu_n^2(r) \sigma^2 + 2 \sigma^4$. Therefore, 
with $Q_i^{(n)} = \chi^2(r,\sigma)$,
\begin{align*}
    2n \leq \ex{F_n | H_1^{(n)}} & = 2n(1-\epsilon_n) + n\cdot\epsilon_n\left(2r\log(n)+\sigma^2\right) = 2n(1 + o(1/\sqrt{n}));
\end{align*}
in the last transition, we used  $\beta>1/2$. Similarly, we have 
\begin{align*}
4n \leq \Var{F_n^2 | H_1^{(n)}} & = 4n (1 - \epsilon_n) + n \cdot \epsilon_n \left( 4 r \log(n) \sigma^2 + 2 \sigma^2 \right) \\
& = 4n(1 + o(1/\sqrt{n})) = 4n(1 + o(1)),
\end{align*}
hence \eqref{eq:proof:Fisher:to_show} holds. 

For the general case of \eqref{eq:hyp_log_n_appx} under \eqref{eq:APLCC_cond}, first note that 
\begin{align}
    \exsub{X \sim E_i^{(n)}}{X} & = \int_{0}^\infty \Prp{X \geq x|X\sim E_i^{(n)}} dx \nonumber\\
    & = 2 \log(n) 
    \int_{0}^\infty \Prp{X \geq 2 y \log(n)|X\sim E_i^{(n)}} dy \nonumber \\
    & =  2 \log(n) 
    \int_{0}^\infty e^{-y \log(n) (1+o(1))} dy = 2(1+o(1)), \label{eq:Fisher:proof1}
\end{align}
and, by a change of variable and evaluation of the integral of the function $y \to ye^{-y}$,  
\begin{align*}
    \exsub{X \sim E_i^{(n)}}{X^2} & = \int_{0}^\infty x\Prp{X \geq x|X\sim E_i^{(n)}} dx \\
    & = (2 \log(n))^2
    \int_{0}^\infty y \Prp{X \geq 2 y \log(n)|X\sim E_i^{(n)}} dy \\
    & =  (2 \log(n))^2
    \int_{0}^\infty y e^{-y \log(n) (1+o(1))} dy \\
    & = \frac{(2 \log(n))^2}{(2 \log(n)(1+o(1))^2} = 4(1+o(1)).
\end{align*}
It follows that 
\begin{align}
\ex{F_n|H_0^{(n)}} = 2n(1+a_n),\quad \text{and}\quad \Var{F_n|H_0^{(n)}} = 4n(1+o(1)),
\label{eq:Fisher:proof_null}
\end{align}
where $a_n \to 0$. Next, notice that 
\begin{align}
    \exsub{X\sim Q_i^{(n)}}{X} & = \int_{0}^\infty \Prp{X \geq x|X\sim Q_i^{(n)}} dx  \nonumber \\
    & = 2 \log(n) \int_0^\infty \Prp{X \geq  2y \log(n) |X\sim Q_i^{(n)}} dy \nonumber \\
    & = 
     \int_0^\infty n^{-\alpha(y;r,\sigma)+o(1)} dy, \label{eq:proof:Fisher:1}  \\
     & = n^{o(1)}, \label{eq:proof:Fisher:2}
\end{align}
where \eqref{eq:proof:Fisher:1} follows from \eqref{eq:APLCC_cond} and \eqref{eq:proof:Fisher:2} is due to
\begin{align*}
\int_0^\infty n^{-\alpha(y;r,\sigma)}dy = 
\frac{\sigma ^2 n^{-\frac{r}{\sigma ^2}}}{\log (n)}+ 2\sigma \sqrt{ \frac{\pi r}{\log(n)}} \Phi \left(\frac{ \sqrt{ 2r \log (n)}}{\sigma }\right)= o(1).
\end{align*}
From \eqref{eq:Fisher:proof1} and because $\beta>1/2$,
\begin{align*}
2n \leq \ex{F_n|H_1^{(n)}} & = 2n(1-\epsilon_n)\left( 1 + a_n) \right) + \epsilon_n n^{1+o(1)}  \\
& \leq 2n(1+a_n) + o(1/\sqrt{n}).
\end{align*}
Similarly, 
\begin{align*}
    \exsub{X\sim Q_i^{(n)}}{X^2} & = \int_{0}^\infty x\Prp{X \geq x|X\sim Q_i^{(n)}} dx  \\
    & = (2 \log(n))^2 \int_0^\infty y \Prp{X \geq  2y \log(n) |X\sim Q_i^{(n)}} dy \\
    & = 
    (2 \log(n))^2 \int_0^\infty y \cdot n^{-\alpha(y;r,\sigma)+o(1)} dy  = n^{o(1)},
\end{align*}
where in the last transition we used that 
\[
\int_0^\infty y\cdot  n^{-\alpha(y;r,\sigma)} dy = o(1),
\]
as can be deduced from the analytic expression of this integral. We obtain
\begin{align}
4n \leq \ex{F_n^2|H_1^{(n)}} & = 4n(1-\epsilon_n)(1+o(1)) + n\cdot \epsilon_n \cdot n^{o(1)} \\
& = 4n(1+o(1)). \label{eq:Fisher:proof:H1_F2}
\end{align}

Evaluations similar to those in  \eqref{eq:Fisher:proof_null} and \eqref{eq:Fisher:proof:H1_F2} imply that $F_n$ satisfies the conditions of the Lyapunov central limit theorem for sums of independent but perhaps non-identically distributed RVs. Consequently, $F_n$ is asymptotically normal both under $H_0^{(n)}$ and $H_1^{(n)}$. Since we have
\begin{align*}
    \frac{\ex{F_n|H_1^{(n)}}-\ex{F_n|H_0^{(n)}} }{\sqrt{\Var{F_n|H_0^{(n)}} } } \to 0,\quad \text{and} \quad  \frac{\Var{F_n|H_0^{(n)}} }{\Var{F_n|H_1^{(n)}} } \to 1,
\end{align*}
we conclude that $F_n$ is asymptotically powerless 
(c.f. \cite[Lem. B.2]{ariascastro2013distributionfree}). 
\qed                                        

\section{Proofs of Results in Section~\ref{sec:models} \label{sec:proofs_of_props}}
\setcounter{equation}{0}

\subsection{Proof of Proposition~\ref{prop:poisson_pvals}}

We use Lemma~\ref{lem:Poisson} with $\lambda_n = \lambda_i$, 
$a_n=2q \log(n)$ and $b_n = 2r \log(n)$. We obtain 
\[
- \log \Prp{-2\log \Pp( X_i; \lambda) \geq 2q\log(n)} = \log(n)\alpha(q;r,1)(1+o(1)),
\]
where $o(1) \to 0$ independently of $\lambda_i$. Proposition~\ref{prop:poisson_pvals} follows. \qed

\subsection{Proof of Proposition~\ref{prop:twosample_Poisson}
\label{sec:prop:twosample_poisson:proof}
}
 Our analysis relies on moderate deviation estimate for variance-stabilized Poisson counts as provided in the following lemma from \citep{DonohoKipnis2020}.
\begin{lemma}{\citep[Lemma 5.3]{DonohoKipnis2020}}\label{lem:modorate_deviation}
Let $\Upsilon_\lambda',\Upsilon_\lambda$ denote two independent Poisson RVs. Let $a(\lambda)$ be a non-negative function. Consider a sequence of pairs $(\lambda, \lambda')$ such that 
$\lambda \to \infty$, $\lambda' \geq \lambda$, $\lambda'/\lambda \to 1$ as $n \to \infty$. Also suppose $a(\lambda) - (\sqrt{2\lambda'} - \sqrt{2\lambda}) \to \infty$ while $a(\lambda)/\lambda \to 0$. Then:
\[
\lim_{n \to \infty} \frac{1}{ \left(\sqrt{a(\lambda)}
-(\sqrt{2\lambda'}-\sqrt{2\lambda})
\right)^2}  \log \left[ \Pr\left( \sqrt{2\Upsilon_{\lambda'}}-\sqrt{2\Upsilon_{\lambda}} \geq \sqrt{a(\lambda)} \right) \right] = -\frac{1}{2}.
\]
\end{lemma} 

Let $W_i := \sqrt{2Y_i}-\sqrt{2X_i}$ and $S_i := -2\log(\pi_i)$. We have
\begin{align*}
     \Prp{S_i > 2q\log(n) } 
    & =  \Prp{2 \bar{\Phi}(W_i) < n^{-q} } \\
    & = \Prp{W_i > \bar{\Phi}^{-1}(n^{-q}/2) } 
\end{align*}
By Mill's ratio, 
\[
\bar{\Phi}^{-1}(n^{-q}/2) = \sqrt{2q \log(n)(1+o(1))}.
\]

Let $\lambda'_i = \lambda_i + \sqrt{\mu_n(r) \lambda_i}$. 
Note that 
$\lambda_i'/\lambda_i = 1 + \sqrt{2r \log(n)/\lambda_i} \to 1$ uniformly in $i\leq n$ by \eqref{eq:lambda_cond}, we have
\[
\sqrt{2 \lambda_i'} - \sqrt{2 \lambda_i} = \sqrt{r \log(n)}(1+o(1))
\]
where here and henceforth $o(1)$ indicates a sequence tending to $\to 0$ uniformly in $i$. Consequently,
\begin{align*}
\sqrt{a(\lambda_i)} - (\sqrt{2\lambda_i'} - \sqrt{2\lambda_i}) & = \sqrt{2q\log(n)(1+o(1))} -  \sqrt{2\lambda_i}\left(1 + o(1) \right) \\
& =: 2\log(n) \alpha(q;r/2,1)(1+o(1)). 
\end{align*}

The above evaluations show that 
$\lambda = \lambda_i$, 
$\lambda' = \lambda_i + \sqrt{2r \log(n) \lambda_i }$ satisfies the conditions of
Lemma~\ref{lem:modorate_deviation}. We obtain
\begin{align*}
      \log\left(\Prp{W_i > \bar{\Phi}^{-1}(n^{-q}/2)} \right) = -\log(n) \alpha(q;r/2,1)(1+o(1)),
\end{align*}
with $o(1)$ independent of $i$. It follows that
\begin{align*} 
\max_{1 \leq i \leq n} \left| \frac{-\log \Prp{ S_i > 2 q \log(n)}}{\log(n)} - \alpha(q;r/2,1) \right| \to 0.
\end{align*}


\subsection{Proof of Proposition~\ref{prop:binomial}}
\label{sec:proof:prop:binomial}

By the symmetry of $\Bin(m,1/2)$ around $m/2$, for $x \geq 0$ we have 
\begin{align*}
p_{\mathsf{Bin}}(x) & = \Prp{\left| \Bin(m,1/2) - \frac{m}{2} \right| \geq \left| x - m/2 \right|} \\& = 2 \Prp{ \frac{\Bin(m,1/2) - \frac{m}{2}}{\sqrt{m}/2}  \geq  \frac{|x - m/2|}{\sqrt{m}/2}}.
\end{align*}
A Berry-Essen-type argument applied to the binomial survival function (e.g. \cite[App. 1]{shorack2009empirical}) implies 
\[
\left|p_{\mathsf{Bin}}(x) - 2 \bar{\Phi}\left( t(x;m) \right) \right| \leq \frac{C_1}{\sqrt{t(x)}},\quad t(x;m) = \frac{|x-m/2|}{\sqrt{m}/2}. 
\]
for some constant $C_1$. Therefore, by Mill's ratio, as $t \to \infty$, 
\begin{align*}
-2\log p_{\mathsf{Bin}}(x) = \left( \frac{x - m/2}{\sqrt{m}/2} \right)^2 + O(1). 
\end{align*}
By standard central limit argument, 
\[
\frac{X - m(1/2+\delta)}{\sqrt{m(1/4-\delta^2)}} \overset{D}{=} Z + o_p(1),\quad Z \sim \Ncal(0,1),
\]
as $m \to \infty$, hence 
\[
t(X;m) + o_p(1) \overset{D}{=} \sqrt{1-4\delta^2} Z + 2\sqrt{m} \delta = \sqrt{1-s r} Z + \sqrt{2 r \log(n)}
\]
is unbounded in probability as $n \to \infty$. We obtain
\begin{align*}
 & \log \Prp{ -2\log p_{\mathsf{Bin}}(X) \geq 2 q \log(n)} = \log \Prp{ \left( \frac{X - m/2}{\sqrt{m}/2} \right)^2 + O_P(1) \geq 2 q \log(n)} \\
 & \qquad  = \log \Prp{  \frac{X - m/2}{\sqrt{m}/2}   \geq \sqrt{2 q \log(n)}(1+o(1)) } \\
 & \qquad = \log \Prp{  \sqrt{1- 4\delta^2} \frac{X - m(1/2+\delta)}{\sqrt{m(1/4- \delta^2)}} +
 2\sqrt{m}\delta \geq \sqrt{2 q \log(n)}(1+o(1)) } \\
 & \qquad = \log \Prp{  \sqrt{1- s r} \frac{X - m(1/2+\delta)}{\sqrt{m(1/4- \delta^2)}} +
 \sqrt{2 r \log(n)} \geq \sqrt{2 q \log(n)}(1+o(1))} \\
 & \qquad = \log \Prp{ \frac{X - m(1/2+\delta)}{\sqrt{m(1/4- \delta^2)}} \geq \sqrt{2 \log(n)}\frac{\sqrt{q} - \sqrt{r}}{\sqrt{1- rs}}(1+o(1))} \\
 & \qquad = \log \Prp{ \frac{X - m(1/2+\delta)}{\sqrt{m(1/4- \delta^2)}} \geq \sqrt{2 \log(n) \alpha(q;r, 1-sr) }(1+o(1))}. 
\end{align*}
From here, \eqref{eq:prop:binomial} follows by
applying Cram\'er's theorem \eqref{eq:cramer}.

\else
\section*{Supplementary Materials}
Contain proofs of all Theorems and Propositions. 
\par
\fi

	
\section*{Acknowledgment}
The author would like to thank David Donoho for discussions and comments on an early version of this manuscript and two anonymous reviewers who provided valuable suggestions that have improved the manuscript. Parts of this article were presented at the 2022 IEEE International Symposium on Information Theory (ISIT) \citep{kipnis2022isit}. Parts of the work were done while the author was with the Department of Statistics at Stanford University.

\bibhang=1.7pc
\bibsep=2pt
\fontsize{9}{14pt plus.8pt minus .6pt}\selectfont
\renewcommand\bibname{\large \bf References}
\expandafter\ifx\csname
natexlab\endcsname\relax\def\natexlab#1{#1}\fi
\expandafter\ifx\csname url\endcsname\relax
  \def\url#1{\texttt{#1}}\fi
\expandafter\ifx\csname urlprefix\endcsname\relax\def\urlprefix{URL}\fi

\bibliographystyle{chicago}      
\bibliography{HigherCriticism}  
\vskip .65cm
\noindent
School of Computer Science, Reichman University
\vskip 2pt
\noindent
E-mail: (\texttt{alon.kipnis@runi.ac.il})
\vskip 2pt

\end{document}